\documentclass[11pt,,eqno,twoside]{article} 
\usepackage{bbm}
\usepackage{mathrsfs}
\usepackage{enumerate}
\usepackage{amssymb}
\usepackage{amsmath}
\usepackage{amsthm}
\usepackage{amsfonts}
\usepackage{color}
\usepackage{graphicx}
\usepackage[active]{srcltx}
\usepackage{CJK}

\usepackage[cp1252]{inputenc}

\usepackage{mathrsfs}
\usepackage{graphicx}

\usepackage[active]{srcltx}

\allowdisplaybreaks

\usepackage{titletoc}
\titlecontents{section}[0pt]{\addvspace{2pt}\filright}
              {\contentspush{\thecontentslabel\ }}
              {}{\titlerule*[8pt]{.}\contentspage}

\def\dd{\mathrm{d}}
\def\rr{{\mathbb R}}
\def\rn{{{\rr}^n}}

\def\nn{{\mathbb N}}

\def\cb{{\mathcal B}}

\def\fz{\infty}

\def\dist{{\mathop\mathrm{\,dist\,}}}
\def\loc{{\mathop\mathrm{\,loc\,}}}

\def\lip{{\mathop\mathrm{\,Lip}}}

\def\lz{\lambda}
\def\dz{\delta}

\def\ez{\epsilon}

\def\kz{\kappa}
\def\bz{\beta}

\def\gz{{\gamma}}

\def\vz{\varphi}

\def\wz{\widetilde}

\def\bint{{\ifinner\rlap{\bf\kern.35em--}
\int\else\rlap{\bf\kern.45em--}\int\fi}\ignorespaces}

\def\bbint{{\ifinner\rlap{\bf\kern.35em--}
\hspace{0.078cm}\int\else\rlap{\bf\kern.45em--}\int\fi}\ignorespaces}

\def\diam{{\mathop\mathrm{\,diam\,}}}

\newtheorem{thm}{Theorem}[section]
\newtheorem{lem}[thm]{Lemma}
\newtheorem{rem}[thm]{Remark}
\newtheorem{cor}[thm]{Corollary}
\numberwithin{equation}{section}

\begin{document}

\arraycolsep=1pt

\title{\Large\bf A global second order Sobolev regularity for
 $p$-Laplacian type equations with  variable coefficients in bounded domains
 \footnotetext{\hspace{-0.35cm}
\endgraf The first author is supported by NSFC (No.12001041 \& No.11871088 \& No. 12171031)
and Beijing Institute of Technology Research Fund Program for Young Scholars.
 The second author is supported by China Postdoctoral Science Foundation funded project
 (No. BX20220328). The third author is supported by NSFC (No. 11871088 \& No.12025102) and by the Fundamental Research Funds for the Central Universities.
\endgraf }
}
\author{Qianyun Miao, Fa Peng,  Yuan Zhou}
\date{ }

\maketitle

\begin{center}
\begin{minipage}{13.5cm}\small
{\noindent{\bf Abstract.}\quad
Let $\Omega\subset\rr^n$ be a bounded convex domain  with $n\ge2$.
Suppose that  $A$ is uniformly elliptic and belongs to $W^{1,n}$ when $n\ge 3$  or $W^{1,q}$ for some $q>2$  when $n=2$.
For $1<p<\fz$,  we build up a global second order regularity estimate
$$\|D[|Du|^{p-2} Du]\|_{L^2(\Omega)}+\|D[ |\sqrt{A}Du|^{p-2} A Du]\|_{L^2(\Omega)} \le C \|f\|_{L^2(\Omega)} $$
for  inhomogeneous $p$-Laplace type equation
\begin{equation*}
  -\mathrm{div}\big(\langle A Du,Du\rangle ^{\frac{p-2}2} A Du\big)=f  \quad\rm{in }\ \Omega \mbox{ with Dirichlet/Neumann $0$-boundary.}
\end{equation*}
Similar result was also built up for certain bounded Lipschitz domain whose boundary is weakly second order differentiable  and satisfies some smallness assumptions.

}

\end{minipage}
\end{center}


\section{Introduction}\label{s1}

We first recall the $L^2$-integrability of second order derivatives (also called the Hessian or Calderon-Zygmund estimate)
 for the Poisson equation in bounded convex domains and its extension to divergence type elliptic equations;
for example \cite{ADN59,Grisvard, Horm63,LadU68,MazS09,A92} and references therein.
For these and also their  extension to more non-smooth domains we refer to for example \cite{ADN59,Grisvard, Horm63,LadU68,MazS09,A92}
and references therein.
To be precise, let $\Omega\subset\rn$  with $n\ge2$ be any  bounded convex domain.
Given  any $f\in L^2(\Omega)$,  the uniqueness  weak solution $u$ to the Poission equation  $$\mbox{$ \Delta u:=-{\rm div}(Du)=f$ in   $\Omega $} $$
 with Dirichlet $0$-boundary $u|_{\partial\Omega}=0$ satisfies
  \begin{equation}\label{secg}     u\in W^{2,2}(\Omega)\quad\mbox{and}\quad\|D^2u\|_{L^{2}(\Omega)}\le C  \|f\|_{L^2(\Omega)}
 \end{equation}
 for some constant $C$ depending on $n$.
Note that \eqref{secg}   also holds whenever $f\in L^2(\Omega)$ 
with $\int_\Omega f(x)\,dx=0$, and
$u$ is any weak solution to $\Delta u=f$ in $\Omega$ with Neumann $0$-boundary  $Du\cdot \nu|_{\partial\Omega}=0$.
Here and below  $\nu$ always denotes the outer  normal direction of the boundary $\partial\Omega$.
Regards of the Neumann $0$-boundary, the condition $\int_\Omega f(x)\,dx=0$ is  necessary  to  the existence of weak solutions,
and weak solutions are unique up to some additive constant.

%
%

Moreover, consider 
 the inhomogeneous  elliptic equation \begin{equation}\label{LA} \mbox{ $\mathcal L_Au:=-{\rm div}\, (A Du) = f$ in    $\Omega$},
\end{equation}
where and below we always suppose that
$A:\Omega\to\rr^{n\times n}$ is a symmetric  matrix-valued function satisfying the elliptic condition
\begin{equation}\label{ell} \frac1{L}|\xi|^2\le \langle A(x)\xi,\xi\rangle\le L|\xi|^2\quad\forall x\in\Omega,\quad \forall \xi\in\rr^n
\end{equation}
for some $L>1$; for short we write $A\in\mathcal E_L(\Omega)$ below. Under  the integrability condition for the distributional derivative $DA$:
\begin{equation} \label{sobA}\mbox{
$D A\in L^{ q} (\Omega, \rr^{n\times n})$ for some $q\ge n$ when $n\ge3$ or $q>n$ when $n=2$,}
\end{equation}
it was  shown  that  
 \eqref{secg} also holds, for some constant $C$ depending on $n, \Omega, q,L$   and $DA$, whenever
  $f\in L^2(\Omega)$  and $u$ is a weak solution to \eqref{LA} with Dirichlet
 $0$-boundary, or whenever   $f\in L^2(\Omega)$ with $\int_\Omega f(x)\,\dd x=0$ and $u$   is a weak solution to \eqref{LA}
   with Neumann $0$-boundary $ADu\cdot\nu|_{\partial\Omega}=0 $. 



Recently, Cianchi-Mazya \cite{CMaz} built up  a nonlinear version of above Hessian or Calderon-Zygmund estimate
 for the inhomogeneous $p$-Laplace equation
\begin{equation}\label{plap}\mbox{$  \Delta_pu:= -\mathrm{div}\left(|Du|^{p-2} Du\right)=f$ in 
$\Omega $}, \end{equation}  where $1<p<\infty$. Indeed, they proved
\begin{equation}\label{czp} |Du|^{p-2}Du\in W^{1,2}(\Omega)\quad\mbox{with the norm  bound}\quad
\|D[|Du|^{p-2}Du]\|_{L^2(\Omega)}\le C\|f\|_{L^2(\Omega)}
\end{equation}
for some  constant $C$ depending on $n,p,\Omega$, whenever
    $f\in L^2(\Omega)$ and $u$ is a generalized solution  to \eqref{plap} with   $u|_{\partial\Omega}=0$ or whenever   $f\in L^2(\Omega)$ with $ \int_\Omega f(x)\,\dd x=0$ and $u$ is a generalized solution  to \eqref{plap} with  $ Du\cdot\nu|_{\partial\Omega}=0$.
Recall that when $p\ne2$,  $\Delta_pu=f$ is a quasilinear degenerate/singular elliptic equation.
Note that when $1<p<2$, since $f\in L^2(\Omega)$ implies $f\in L^p(\Omega)$, the generalized solutions coincides with weak solutions;
  when $2<p<\fz$,  since $f\in L^2(\Omega)$ does not guarantee  $f\in L^p(\Omega)$, one can not define weak solutions via integration by parts,
the it is natural to work with generalized solutions; 
for more about generalized solutions see \cite{CMaz2017,CMaz}. These results were extended by \cite{CMaz2019,bcdm} to vector-valued case.

In this paper we consider the
 inhomogeneous $p$-Laplace type equation
\begin{equation}\label{p-Lap-sys}
\mathcal L_{A,p}u:=
  -\mathrm{div}\big(\langle A Du,Du\rangle ^{\frac{p-2}2} A Du\big)=f\quad\rm{in }\ \Omega,
\end{equation}
where the coefficient    $A\in \mathcal E_L(\Omega)$.  
Motivated by  above Hessian or Calderon-Zygmund estimates  
and their nonlinear version, it is  natural to ask,  under  the condition \eqref{sobA},
 whether   some similar global second order regularity  holds for generalized solutions to  \eqref{p-Lap-sys} with Dirichlet/Neumann 0-boundary.


 The main purpose of this paper is to answer the above question as below.

\begin{thm}\label{thm:convex}
Let $\Omega$ be a bounded convex domain of $\rn$ with $n\ge2$. Let $1<p<\fz$.
Suppose that  $A\in \mathcal E_L(\Omega)$ for some $L>1$  and   satisfies \eqref{sobA}. 

If $f\in L^{2}(\Omega)$ and $u$ is a generalized solution  to
 \eqref{p-Lap-sys} with $u|_{\partial\Omega}=0$, or if $f\in L^{2} (\Omega)$ with $\int_\Omega f(x)\,\dd x=0$ and $u$ is a generalized solution  to
 \eqref{p-Lap-sys} with $ADu\cdot\nu|_{\partial\Omega}=0$,  then
 we have
\begin{equation}\label{2nd-reg}
\left\{
\begin{split}&
\mbox{$|\sqrt{A}Du|^{p-2}ADu\in W^{1,2}(\Omega)$ and $| Du|^{p-2} Du\in W^{1,2}(\Omega)$
with   norm bounds }\\
& \quad\quad
 \|D[|\sqrt{A}Du|^{p-2}ADu]\|_{L^2(\Omega)}+ \|D[| Du|^{p-2} Du]\|_{L^2(\Omega)}\le C\|f\|_{L^{2}(\Omega)},
\end{split}\right.
\end{equation}
where   $C>0$ is a constant independent of $u$ and $f$.
\end{thm}


Moreover, beyond  convex domains,
Cianchi-Mazya \cite{CMaz} proved that the above estimate \eqref{czp} holds for the equation \eqref{plap} in bounded Lipschitz domain $\Omega $ provided that
 \begin{equation}\label{exx1}\mbox{ the boundary $\partial\Omega\in W^2L^{n-1,\infty}$ when $n\ge3$
and $\partial\Omega\in W^2L^{1,\infty}\log L$ when $n=2$}\end{equation}
and that the weak second fundamental form $\mathcal{B}$ on $\partial\Omega$ satisfies
\begin{equation}\label{smallness}\lim_{r\to0}\Psi_\mathcal{B} (r) \le c\quad  \mbox{for some suitable small constant $c=c(\lip_\Omega, d_\Omega, n, p)$,}
\end{equation} where
\begin{equation}\label{B-asumm1}
   \Psi_ \mathcal{B}  (r)  :=\left\{\begin{array}{lc}\displaystyle
  \sup_{x\in\partial\Omega}\|\mathcal{B}\|_{L^{n-1,\infty}(\partial\Omega\cap B_r(x))}  \quad&\textrm{if}\quad n\geq 3,\\\displaystyle
   \sup_{x\in\partial\Omega}\|\mathcal{B}\|_{L^{n-1,\infty}\log L(\partial\Omega\cap B_{r}(x))} \quad&\textrm{if}\quad n=2.
\end{array}\right.
\end{equation}
Here $d_\Omega$ is the diameter of $\Omega$ and $\lip_\Omega$ is the Lipschitz constant of $\Omega$.
Recall that $\Omega$ is  a  Lipschitz domain, if, in a neighborhood $B(x,r) \cap\partial\Omega$ of each boundary point $x$,
$\Omega$ agrees with the subgraph $\Gamma(\phi)$ of a Lipschitz continuous function $\phi:\rr^{n-1}\to\rr$  of $(n-1)$-variables.
We say $\partial \Omega$ satisfies \eqref{exx1} if each such $\phi$ is twice weakly differentiable, and that its second-order derivatives belong to either the weak Lebesgue space $L^{n-1,\infty}$ if $n\geq 3$,
or the weak Zygmund space $L^{1,\infty}\log L$ if $n=2$. 
The assumption  \eqref{exx1} guarantee that  the weak second fundamental form $\mathcal{B}$ on $\partial\Omega$
belongs to the same weak type spaces with respect to the $(n-1)$-dimensional Hausdroff measure $\mathcal{H}^{n-1}$ on $\partial\Omega$,
and  hence    guarantee  $\lim_{r\to0}\Psi_ \mathcal B(r)<\fz$.
But   \eqref{exx1}  cannot give the smallness  of  $\Psi_ \mathcal B(r)$, that is, \eqref{smallness}.
As revealed by Cianchi-Mazya,  to get \eqref{secg} in the case $p=2$ and also \eqref{czp} for $1<p<\fz$
the smallness assumption \eqref{smallness} for $\lim_{r\to0}\Psi_ \mathcal{B}(r)$   is necessary and optimal;
\eqref{smallness}  can not be weaken to $\lim_{r\to0}\Psi_ \mathcal{B} (r)<\fz$. For more details we refer to \cite{CMaz}.


In this paper, we also show that  the convexity assumption on  $\Omega$ in Theorem 1.1 can be  reduced to the smallness  assumption \eqref{smallness} as in \cite{CMaz}.

\begin{thm}\label{thm}
Let $\Omega$ be a bounded  Lipschitz
 domain of $\rn$ with $n\ge2$ and  satisfy  \eqref{exx1}.  Let $1<p<\fz$.
Suppose that $A\in \mathcal E_L(\Omega)$ for some $L>1$  and  satisfies \eqref{sobA}. 

There exists a constant $\dz_\ast>0$ depending on $n,L$ such that if   $\lim_{r\to0}\Psi_ \mathcal B(r) \le \dz_\ast$,
then   \eqref{2nd-reg}   holds  whenever
    $f\in L^2(\Omega)$  and $u$ is a generalized  solution   to
  \eqref{p-Lap-sys} with
 $u|_{\partial\Omega}=0$ or whenever   $f\in L^2(\Omega)$ with $ \int_\Omega f(x)\,\dd x=0$ and $u$ is a generalized  solution   to
  \eqref{p-Lap-sys} with  $ ADu\cdot\nu|_{\partial\Omega}=0$.

\end{thm}


 In   the statements of Theorems \ref{thm:convex} and \ref{thm}, the annoying dependence of constant $C$ on the main parameters  is unclear.
Here we clarify them  respectively.

\begin{rem}\label{rem1.3}\rm
 (i) Regards of  the constant $C$  in Theorem \ref{thm:convex}, we consider two cases.

 \begin{enumerate}

\item[$\bullet$] In the case  $A\in \dot W^{1,q}(\Omega) $ for some   $q>n\ge2$,
the constant $C $  in Theorem 1.1  depends  only on $n,p$, $L$, $d_\Omega=\diam \Omega$, $ |\Omega|$ and $C_{{\rm ext},q}(\Omega)\|DA\|_{L^q(\Omega)}$.
In the proof, we extend $A\in \dot W^{1,q}(\Omega) $ as   $\wz A\in \dot
W^{1,q}(\rn)$ so that $ \|D\wz A\|_{L^q(\rn)}\le C_{{\rm ext},q}(\Omega)\|DA\|_{L^q(\Omega)}$,
where  $C_{{\rm ext},q}(\Omega)$ is the norm of the extension operator.

\item[$\bullet$]  In the case $A \in W^{1,n}({\Omega})$ with   $n\ge 3$,
we know that  $\lim _{r\to0} \Phi_{ A,\Omega}(r)=0$, where   $$\Phi_{A,\Omega}( r):=  \sup_{x\in\overline \Omega}\|DA\|_{L^n(B(x,r)\cap\Omega)}.$$
We  extend $A\in \dot W^{1,n}(\Omega) $ to   $\wz A\in \dot W^{1,n}(\rn)$ so that
$$\Phi_ {\wz A,\Omega^t}(r)\le  C_{{\rm ext},n}(\Omega) \Phi_ {A,\Omega }(C_{{\rm ext},n}(\Omega) r) \quad \mbox{whenever $0<t<r$},$$  where  $C_{{\rm ext},n}(\Omega)$ is a constant and
$\Omega ^t$ is the $t$-neighbourhood of $\Omega$.  To obtain   \eqref{2nd-reg}  in Theorem \ref{thm:convex}, we  require that
$ \Phi_ {A,\Omega }( C_{{\rm ext},n}(\Omega)  r_A)\le \dz_\sharp/ C_{{\rm ext},n}(\Omega) $  for  some  sufficiently small $\dz_\sharp>0$  depending on $n,p,L$  and some $r_A\in(0,1)$.
The constant $C$  in  Theorem \ref{thm:convex} depends only on $n,p,L$ $r_A$, $\lip_\Omega$,  $d_\Omega$ and $|\Omega|$.

\end{enumerate}



(ii) 
The constant $C$  in Theorem \ref{thm} not only depends on    parameters   as stated in (i) above in a similar way, but also depends on
some constant caused by the condition \eqref{smallness}, that is, depends on $1/r_\Omega$, where  $r_\Omega\in(0,1)$  satisfies
$ \Psi_{\mathcal B}(r_\Omega)\le \dz_\ast$  and $\dz_\ast$ is given in \eqref{smallness}.


%
%
%
\end{rem}


To prove Theorems \ref{thm:convex} and \ref{thm}, 
we consider some regularized equation of   \eqref{p-Lap-sys}, that is, the equation
\begin{equation}\label{app-ellip-syseps0}
 \mathcal L_{\ez,A,p}u=-\mathrm{div}\,\left((|\sqrt{A}Du|^2 + \epsilon)^{\frac{p-2}{2}} ADu\right)=f\ \rm{in }\ \Omega \mbox{ with
$u|_{\partial\Omega}=0$   or $ADu\cdot\nu|_{\partial\Omega}=0$},
\end{equation}
where $\ez\in(0,1]$ and   $\Omega,A,f$ satisfies

(S1) $A\in   \mathcal E_L(\rn)\cap  C^\fz(\rn)$,

(S2) $\Omega$ is a bounded smooth  domain,


(S3) $f\in C^\fz_c( \Omega) $ (and additionally, $\int_\Omega f(x)\,dx=0 $  in the case Neumann 0-boundary).

\noindent
Note that,  under assumptions (S1)-(S3),
  weak solutions   to the  regularized equation \eqref{app-ellip-syseps0} with Dirichlet/Neumann 0-boundary are always smooth in $\overline \Omega$.

We  establish the following global quantitative second order regularity
for the regularized equation \eqref{app-ellip-syseps0}    in Theorems \ref{thm:convex-smooth} and \ref {thm:smooth}.
From them, via an standard approximation argument  we
 conclude   Theorems \ref{thm:convex} and
\ref{thm}  (including  
Remark \ref{rem1.3}) respectively; we refer to Section 7 for details.

\begin{thm}\label{thm:convex-smooth}
Let $1<p<\fz$, $\ez\in(0,1]$, and  $\Omega$ be a bounded convex domain.
  Suppose  that  $\Omega$, $A$ and $f$ satisfy assumptions (S1)-(S3).
Let $u$ be  any weak solution    to
\eqref{app-ellip-syseps0}.

\begin{enumerate}
\item[(i)]  If $\|DA \|_{L^q(\Omega)}\le R_\sharp $ for some $q>n\ge 2$ and $0<R_\sharp<\fz$, then
  \begin{equation}\label{2nd-reg-es0}
  \|D[( Du|^2+\epsilon)^{\frac{p-2}{2}} Du]\|_{L^{ 2}(\Omega)}+ \|D[(|\sqrt{A}Du|^2+\epsilon)^{\frac{p-2}{2}}ADu]\|_{L^{ 2}(\Omega)}\leq C\|f\|_{L^2(\Omega)}
\end{equation}
with  the constant
   $C$   depending only on $n,p$, $q$, $L$ and $R_\sharp$.

\item[(ii)]  There exists a constant $\delta_\sharp>0$ depending on $n,p,L ,d_\Omega,\lip_\Omega
$ such that
if $\Phi_{A,\Omega}(r_\sharp)\le \delta_\sharp  $
for some $r_\sharp>0$, then  \eqref{2nd-reg-es0} holds with    the constant
 depending only on $n,p$, $L$ and $r_\sharp$.
\end{enumerate}
\end{thm}


\begin{thm}\label{thm:smooth}
 Let $1<p<\fz$ and $\ez\in(0,1]$.   Suppose  that  $\Omega$, $A$ and $f$ satisfy
  (S1)-(S3).
Let $u$ be a weak solution to
\eqref{app-ellip-syseps0}.

\begin{enumerate}
\item[(i)]  There exists a constant $\dz_\ast>0$ depending only on $n,p,L,\lip_\Omega, d_\Omega$ such that
if   $\Psi_{\mathcal B}(r_\ast)\le \dz_\ast$ for some $0<r_\ast<1$, and
if
$\|D A\|_{L^q(\Omega)}\le R_\sharp$ for some $q>n$ and $0<R_\sharp<\fz$, then
 \eqref {2nd-reg-es0} holds  with the constant
  $C$  depending  only on $n,p$, $L $ and $R_\sharp $ and $r_\ast$ .

\item[(ii)]   There exist  a constant  $\dz_\ast>0$  depending only on $n,p,L,\lip_\Omega,d_\Omega,r_\Omega$,  and  another constant $\delta_\sharp>0$  depending only on $n,p,L,\lip_\Omega,d_\Omega $,
 such that  if
$\Psi_{\mathcal B}(r_\ast)\le \dz_\ast$ for some $0<r_\ast<1$  and if $\Phi_{A,\Omega}(r_\sharp)\le \delta_\sharp$ for some $r_\sharp>0$,  then  \eqref {2nd-reg-es0} holds with
  the constant
  $C$  depending only on $n,p$, $L$, $r_\ast$ and $r_\sharp$.
\end{enumerate}
\end{thm}


 We prove  Theorem \ref{thm:convex-smooth} and  \ref{thm:smooth}  in Section 2  with the aid of key Lemma \ref{L21}, Lemma \ref{L2}, Lemma \ref{boundaryterm}, Lemmas \ref{lem:Veps-es1} and \ref{cmineq2},
whose proofs are postponed to Section 3,  Section 4, Section 5 and  Section 6 correspondingly.

The main  novelty is that, instead of Euclidean gradient $D$, we consider the intrinsic (Riemannian) gradient $\sqrt A D$, which allows
 us to combine the approach of Cianchi-Mazya for the equation $ \Delta_pu=f$ and also the classical approach to Calderon-Zygmund estimates for the equation $\mathcal L_Au=f$.
Moreover, we use different analytic properties of the boundary  of domains when dealing with Dirichlet $0$-boundary and Neumann $0$-boundary; see key Lemma \ref{L21}.



\section{Proofs of  Theorems \ref{thm:convex-smooth}\&\ref{thm:smooth}}

Let $1<p<\fz$ and $\ez\in(0,1]$.   Suppose  that  $\Omega$, $A$ and $f$ satisfy
  (S1)-(S3). Let $u$ be a weak solution to
\eqref{app-ellip-syseps0}.  We are going to  prove   \eqref{2nd-reg-es0}  and then Theorems \ref{thm:convex-smooth}\&\ref{thm:smooth}.
For simplicity of notation, we always write
 $$V_{ A,\ez }(    Du) = (|\sqrt{A}Du|^2+\epsilon)^{\frac{p-2}{2}}ADu\quad\mbox{and}\quad U_{A,\ez}(Du)=(|\sqrt A Du|^2+\ez)^{\frac{p-2}2}|\sqrt AD^2u\sqrt A|.$$
and also $$V_{ \ez }(    Du)= (| Du|^2+\epsilon)^{\frac{p-2}{2}} Du  \quad\mbox{and} \quad U_{ \ez}(Du)=(|Du|^2+\ez)^{\frac{p-2}2}|D^2u|.$$
 Then \eqref{2nd-reg-es0} reads  as
 \begin{equation}\label{end-reg-es0-1}\|DV_{A,\ez}(Du)\|_{L^2(\Omega)}+\|DV _{\ez}(Du)
\|_{L^2(\Omega)}\le C\|f\|_{L^2(\Omega)}.
 \end{equation}
Note that   $$|DV _{\ez}(Du)|\le C(n,p)  U _{\ez}(Du) \le C(n,p,L) U_{A,\ez}(Du),$$
we only need to show
 \begin{equation}\label{end-reg-es0-2}\|DV_{A,\ez}(Du)\|_{L^2(\Omega)}+\|U_{A,\ez}(Du)
\|_{L^2(\Omega)}\le C\|f\|_{L^2(\Omega)}.
 \end{equation}


We proceed as below to prove \eqref{end-reg-es0-2}.
Firstly, we establish  the  following fundamental inequality. 



\begin{lem}\label{L21}
 One has
\begin{equation}\label{div-inequal}
f^2=(\mathcal L_{\ez,A,p}u)^2   \ge \frac12\min\{1,(p-1)^2\} [U_{A,\ez}(Du)]^2 - C  |DA|^2|V_{A,\ez}(Du)|^2
+{\bf I}\\
\end{equation}
where
\begin{equation}\label{I1}{\bf I}=  \mathrm{div}\,\big\{(|\sqrt{A}Du|^2+\epsilon)^{p-2}[\mathrm{tr}(AD^2u)ADu- AD^2uADu ]\big\},
\end{equation}
or
\begin{equation}\label{I2}{\bf I}=  \mathrm{div}\,\big\{(|\sqrt{A}Du|^2+\epsilon)^{p-2}[\mathrm{div}\,(ADu)ADu-(ADu\cdot D )ADu]\big\}.\end{equation}
Here  $C >1$ is  a constant depending on $n,p,L$.
\end{lem}

We introduce different $ {\bf I}$ as in \eqref{I1} and \eqref{I2} and, later in Lemma \ref{boundaryterm}, will use them for Dirichlet and Neumann $0$-boundary respectively.
This is  crucial and also necessary for us to
get Theorem \ref{thm:convex-smooth}\& \ref{thm:smooth} (and hence Theorem \ref{thm:convex}\& \ref{thm}) under the merely regularity assumption   $DA\in L^q(\Omega) $ 
as in (S2).
See Remark \ref{rem-K}(i)  for detailed reasons.

We prove \eqref{div-inequal} with {\bf I} given by \eqref{I2} and $\frac12 $ replaced by $\frac34$ via  considering the intrinsic (Riemannian) gradient $D_Au=\sqrt{A}Du$ and  borrowing some ideas from Mazya-Cianchi \cite[Lemma ]{CMaz}, see Lemma 3.1  for the details.
The additional term  $C|DA|^2V_{A,\ez}(Du)^2$ in \eqref{div-inequal} appears in a natural way.
One may use
  a similar argument to prove  \eqref{div-inequal} with {\bf I} given by \eqref{I1}. But
 instead, we bound  the difference between \eqref{I1} and \eqref{I2}    by
 $$\dz U_{A,\ez}(Du)^2  + \frac1\dz C|DA|^2V_{A,\ez}(Du)^2\quad\mbox{
for any $\dz>0$,  }$$
 see Lemma 3.2 for details.
This allows us to get \eqref{div-inequal}  with ${\bf I}$ given by \eqref{I1}.

In the special case $ A=I_n$,  our proof does give
 $$ [{\rm div}((|Du|^2+\ez)^{\frac{p-2}2}Du)]^2\ge  \min\{1,(p-1)^2\} (| Du|^2+\epsilon)^{p-2}| D^2u|^2-{\bf I},
$$
where     $$
 {\bf I}=\mathrm{div}\,\big\{(|Du|^2+\epsilon)^{p-2}[\Delta u Du-  D^2uDu ]\big\},   $$
no  matter which is given by \eqref{I1} or \eqref{I2}; see Remark 3.6.
This inequality  with the coefficient $\min\{1,(p-1)^2\}$ replaced by some    $\kz >0$ was first proved by Mazya-Cianchi \cite{mc}.
When $1<p<2$, our argument gets a clear coefficient $(p-1)^2$ and also  simplifies their original argument technically; see Remark 3.6 for more details.
Moreover, the key  inequality \eqref{I2-eq12}
used by \cite{CMaz} and also here can be proved in a simple way; see Remark 3.5 and   Lemma 3.3.

 Multiplying both sides of  \eqref{div-inequal} by some test functions and integrating,
 we could conclude the following Lemma. Note that the boundary term    ${\bf K}(\phi)$  as in \eqref{K1}  below
and  the  boundary term  ${\bf K}(\phi)$ as in \eqref{K2}  below
come,  respectively, from  \eqref{I1}   and  from \eqref{I2}.

\begin{lem}\label{L2}  
For any $\phi\in C_c^\fz(\rn)$, one has
\begin{align}\label{e3.x3}
& \|\phi  DV_{A,\ez}(Du) \|^2_{L^2(\Omega)} + \|U_{A,\ez}(Du) \|^2_{L^2(\Omega)}\nonumber\\
&\le C\|\phi f\|^2_{L^2(\Omega)} +C _1\||DA|\phi V_{A,\ez}(Du)\|^2_{L^2(\Omega)}  +C_2 \||D\phi| V_{A,\ez}(Du)\|^2_{L^2(\Omega)}  +{\bf K}(\phi)
\end{align}
%
where
\begin{equation}\label{K1}
  {\bf K}(\phi) = 
  - C
\int_{\partial\Omega}\phi^2(|\sqrt{A}Du|^2+\epsilon)^{p-2} [ \mathrm{tr}(AD^2u)ADu - AD^2uADu ]\cdot\nu\dd\mathcal{H}^{n-1}(x),\\
\end{equation}
or
\begin{equation}\label{K2}
{\bf K}(\phi)  =-C\int_{\partial\Omega}\phi^2(|\sqrt{A}Du|^2+\epsilon)^{p-2} [\mathrm{div}(ADu)ADu -  (ADu\cdot D )ADu ]\cdot\nu\dd\mathcal{H}^{n-1}(x).
\end{equation}
Here $C$, $C_1$ and $C_2$  are positive constants depending  only on  $n, p, L$.
 \end{lem}



Thanks to
 Lemma \ref{L2}, to get \eqref{end-reg-es0-1}  we only  need to bound the last three terms  in the right hand side of \eqref{e3.x3}, that is,
$$  C_1\||DA|\phi V_{A,\ez}(Du)\|^2_{L^2(\Omega)} ,\quad C_2 \||D\phi| V_{A,\ez}(Du)\|^2_{L^2(\Omega)}, \quad\mbox{and}\quad {\bf K}(\phi).$$

The boundary term ${\bf K}(\phi)$ is bounded as in Lemma \ref{boundaryterm}.
\begin{lem}\label{boundaryterm}

(i) If $\Omega$ is bounded smooth convex domain, then    ${\bf K}(\phi)\le 0$ whenever $\phi\in C_c^\fz(\rn)$.

(ii) If $\Omega$ is bounded smooth   domain,   then   \begin{align*}
   {\bf K}(\phi)
 \le C_\ast \Psi_{\mathcal B}(r)[\|\phi   D V_{A,\ez}(Du) \|_{L^2(\Omega)}^2+ \||D \phi   | V_{A,\ez}(Du)\|_{L^2(\Omega)}^2]
\end{align*}
 whenever $\phi\in C_c^\fz(B(z,r))$ with  $z\in\overline \Omega$ and $0<r<1$. Here $C_\ast>0$ is a constant depending only on $ n,L$.
\end{lem}

See Section 5 for the proof of Lemma \ref{boundaryterm}.
Here are some necessary remarks for the proof.
\begin{rem}\label{rem-K}\rm
(i) When $u$ has Dirichlet  $0$-boundary, to get  some suitable  estimate of ${\bf K}(\phi)$ as above,
we have to use  ${\bf K}(\phi)$ as given in \eqref{K1}, which comes from \eqref{I1}.
Otherwise, if we use the  ${\bf K}(\phi)$ as given in \eqref{K2},
then $ DA |_{\partial\Omega}$ will appear in the upper bound of ${\bf K}(\phi)$. However,  no  assumption is made on $ DA |_{\partial\Omega}$ in this paper;
 $DA\in L^q(\Omega)$ as in (S2) does  not give any information of $DA|_{\partial\Omega}$.

Similarly, when $u$ has Neumann $0$-boundary, to get   some suitable  estimate of ${\bf K}(\phi)$ as above,
we also have to use  ${\bf K}(\phi)$ as given in \eqref{K2}, which comes from \eqref{I2}.

(ii)  Besides (i), the proof of Lemma \ref{boundaryterm} (i) relies on the convex geometry;
the proof of Lemma \ref{boundaryterm}  (ii) relies on trace formula by Cianchi-Mazya \cite{CMaz}.
 Some careful calculation/obervation are also necessary.

\end{rem}

Concerning the  term $C\||DA|\phi V_{A,\ez}(Du)\|^2_{L^2(\Omega)} $,
we use the Gargliardo-Nirenberg-Sobolev inequality to get the following upper bound via the $L^q$-norm of $DA$ with $q> n\ge 2$, and  also local $L^n$-norm of $DA$ with $n\ge3$; for the proof see  Section 6.

\begin{lem}\label{lem:Veps-es1}
(i) Given  $q>n$, for any $0<\eta<1$  we have
  \begin{align*}
 \|DA||V_{A,\ez}(Du)|\|_{L^2(\Omega)}^2
  &\le   \eta\|DA\|^2_{ L^{q }(\Omega  )}   \|  D  V_{A,\ez}(Du)  \|^2_{L^2(\Omega)} + \frac C\eta \|DA\|^2_{ L^{q }(\Omega  )}
   \|  V_{A,\ez}(Du) \|^2_{L^1(\Omega)}.
  \end{align*}

(ii)  Given any $\phi\in C_c^\fz(B(z,r))$ for some $z\in\overline\Omega$ and $0<r<1$, we have
  \begin{align*}
   \||DA|\phi|V_{A,\ez}(Du)|\|_{L^2(\Omega)}^2
  &\le  C_\sharp\|DA\|^2_{ L^{n }(\Omega \cap B(z,r))}   \|\phi  D  V_{A,\ez}(Du)  \|^2_{L^2(\Omega)}\\
&+ C_3\|DA\|^2_{ L^{n }(\Omega \cap B(z,r))}   \big[
  \| | D\phi    |V_{A,\ez}(Du) \|^2_{L^2(\Omega)} + C   \|\phi  V_{A,\ez}(Du) \|^2_{L^1(\Omega)}\big].
  \end{align*}
\end{lem}
Towards the  term $C\||D\phi| V_{A,\ez}(Du)\|^2_{L^2(\Omega)} $,
similarly to Lemma \ref{lem:Veps-es1}, one has 
the following.
\begin{lem}\label{cmineq2}  For any $\eta>0$    we have
  \begin{align*}
 \|V_{A,\ez}(Du)\|_{L^2(\Omega)}^2
  &\le \eta \|D V_{A,\ez}(Du)\|_{L^2(\Omega)}^2+ C  \|V_{A,\ez}(Du)\| ^2_{L^1(\Omega)}.
  \end{align*}
\end{lem}

Observe that   the $L^1$-norm of $V_{A,\ez}(Du)$ or $\phi V_{A,\ez}(Du)$ appears in Lemma 2.4 and Lemma 2.5.
To handle them, we need the following $L^1$-estimate by   \cite{CMaz2017}:
\begin{equation}
  \|(|Du|^2 + \epsilon)^{\frac{p-2}{2}} |Du|  \|_{L^1(\Omega)}  \leq C \|f\|_{L^1(\Omega)}.
\end{equation}
 As a consequence we have
\begin{cor}\label{L1-esti}  We have $
 \|V_{A,\ez}(Du)\|^2_{L^1(\Omega)}   \leq C \|f\|^2 _{L^2(\Omega)}.
$
\end{cor}

 Now, with   Lemma \ref{L2}-Lemma \ref{lem:Veps-es1} and Corollary \ref{L1-esti} in hand, we conclude \eqref {end-reg-es0-1}  as below.

If $\Omega$ is convex,  Lemma \ref{boundaryterm} (i) gives ${\bf K}(\phi) \le0 $.
\begin{enumerate}
\item [(i)] Case  $DA\in L^q(\Omega)$ with $q>n$, we  choose  a  test function $\phi$ with  $\phi=1$ in $\Omega$ so that   $C\||D\phi| V_{A,\ez}(Du)\|^2_{L^2(\Omega)} =0$.
Then  \eqref {end-reg-es0-1} from Lemma \ref{L2} and Lemma \ref{lem:Veps-es1} (i) with sufficiently small $\eta>0$ so that  the coefficient $ \eta\|DA\|_{L^q(\Omega)}$    is small.
This give Theorem \ref{thm:convex-smooth}(i).
\vspace{-0.8cm}
\item [(ii)] Case $DA\in L^n(\Omega)$ with $n\ge 3$, we  find small $0\le r_\sharp\le 1$ so that the coefficient
 $C_\sharp\|DA\|^2_{ L^{n }(\Omega \cap B(z,r))}$ appeared in Lemma 2.4(ii) is small.
Then cover $\Omega$ by a family of balls $B_k$ with radius $r_\sharp/4\le r_k<r_\sharp $, and  denote by  $\{\phi_k\}$ be an associated  unity of partition.
Apply  \eqref{e3.x3} to such such $\phi_k$, and use  Lemma 2.4(ii) to bound $ \||DA|\phi_k V_{A,\ez}(Du)\|^2_{L^2(\Omega)}$.
Summation over all $\phi_k$ and using lemma \ref{lem:Veps-es1} we get the desired upper bound as in Theorem \ref{thm:convex-smooth}(ii).
See Section 7 for details.
\end{enumerate}

For general domain $\Omega$, we use Lemma \ref{boundaryterm} (ii) to bound ${\bf K}(\phi)$.
We find $0<r_\ast\le 1$ so that  the coefficient $ C_\ast \Psi_{\mathcal B}(r)$ appeared in Lemma \ref{boundaryterm} is sufficiently small.
Then cover $\Omega$ by a family of balls $B_k$ with radius $r_\ast/4\le r_k<r_\ast $, and  denote by  $\{\psi_k\}$ be an associated  unity of partition.
Apply  \eqref{e3.x3} to such such $\psi_k$, and apply Lemma \ref{lem:Veps-es1}(ii) to bound $ \||DA|\psi_k V_{A,\ez}(Du)\|^2_{L^2(\Omega)}$. Summation over all $\psi_k$ we obtain
$L^2$-norm of $DV_{A,\ez}(Du)$ is bounded by the summation of the $L^2$-norm of  $f$ ,$|DA|V_{A,\ez}(Du)$  and $V_{A,\ez}(Du)$, see Lemma \ref{lem2.8} below.
\begin{enumerate}
\item [(iii)] Case  $DA\in L^q(\Omega)$ with $q>n$,
then \eqref {end-reg-es0-1}  follows from this and Lemma \ref{boundaryterm} (ii), Lemma \ref{lem:Veps-es1} and Corollary \ref{L1-esti}.  This gives Theorem \ref{thm:smooth}(i).
\vspace{-0.3cm}
\item [(iv)] Case $DA\in L^n(\Omega)$ with $n\ge 3$,
we need to localized Lemma \ref{lem2.8} via a unit of partition  as in the proof of Theorem \ref{thm:convex-smooth} (ii). Then using the argument therein we derive \eqref{end-reg-es0-1}.
This gives Theorem \ref{thm:smooth}(ii).
See Section 2.2 for details.
\end{enumerate}

%

 \subsection{Proof of  Theorem \ref{thm:convex-smooth}}
Suppose that $\Omega$ is convex.  By Lemma 2.2(i), we have  ${\bf K}(\phi)\le 0$. Thus \eqref{e3.x3} gives
\begin{align}\label{e3.x3-con}
& \|\phi  DV_{A,\ez}(Du) \|^2_{L^2(\Omega)} + \|\phi U_{A,\ez}(Du) \|^2_{L^2(\Omega)}\nonumber\\
&\le C\|\phi f\|^2_{L^2(\Omega)} +C_1 \||DA|\phi V_{A,\ez}(Du)\|^2_{L^2(\Omega)}  +C_2 \||D\phi| V_{A,\ez}(Du)\|^2_{L^2(\Omega)}
\end{align}

\begin{proof} [Proof of Theorem \ref{thm:convex-smooth} (i).]
 Choose $\phi\in C^\fz_c(\rn)$ such that $\phi=1$ in $\Omega$. Then $D\phi=0$ in $\Omega$ and hence \eqref{e3.x3-con} reads as
\begin{align}\label{e3.x3-con1}
& \|   DV_{A,\ez}(Du) \|^2_{L^2(\Omega)} + \|U_{A,\ez}(Du) \|^2_{L^2(\Omega)} \le C\| f\|^2_{L^2(\Omega)} +C_1 \||DA| V_{A,\ez}(Du)\|^2_{L^2(\Omega)}
\end{align}

By Lemma \ref{lem:Veps-es1}  (i) and choosing $\eta$ sufficiently small such that
  $  C_ 1\eta\|DA\|^2_{ L^{q }}\le \frac14$, we obtain
  \begin{align*}
C_1 \||DA| V_{A,\ez}(Du)\|^2_{L^2(\Omega)}
&\le \frac14 \|   DV_{A,\ez}(Du) \|^2_{L^2(\Omega)}+ C \|DA\|^2_{ L^{q }}
\|V_{A,\ez}(Du)\| ^2_{L^1(\Omega)}.
  \end{align*}
Thanks to  this and  $ \|V_{A,\ez}(Du)\|_{L^1(\Omega)}^2\le   C\|f\|_{L^2(\Omega)}$ given in Corollary \ref{L1-esti},  from \eqref{e3.x3-con1} we  conclude
$$
 \frac34 \|   DV_{A,\ez}(Du) \|^2_{L^2(\Omega)} + \| U_{A,\ez}(Du) \|^2_{L^2(\Omega)}\nonumber\\
 \le C\| f\|^2_{L^2(\Omega)},
$$
that is, \eqref{end-reg-es0-1} holds.
\end{proof}
\begin{proof}[Proof of Theorem \ref{thm:convex-smooth}(ii).]
Since $DA\in L^n(\Omega)$, then there exists $r_\sharp>0$ depends on $n$ and $A$ such that
 $$C_1C_\sharp\|DA\|^2_{ L^{n }(\Omega \cap B(x,r_\sharp))}\le \frac14,\quad\quad\forall x\in\overline\Omega,$$
where $C_\sharp$ is as in Lemma \ref{lem:Veps-es1} (ii).
By this inequality and Lemma \ref{lem:Veps-es1}  (ii),  for any
  $\phi\in C^\infty_c(B(x,r_\sharp))  $,  one has
  \begin{align}\label{xxx1}
   C_1 \||DA|\phi V_{A,\ez}(Du)\|^2_{L^2(\Omega)}
  &\le  \frac14  \|\phi  D  V_{A,\ez}(Du)  \|^2_{L^2(\Omega)}\nonumber\\
&+     \frac{C_3}{4C_\sharp} [
  \| | D\phi    |V_{A,\ez}(Du) \|^2_{L^2(\Omega)} + C   \|\phi  V_{A,\ez}(Du) \|^2_{L^1(\Omega)}].
  \end{align}
Inserting this inequality into \eqref{e3.x3-con} yields
\begin{align}\label{g2}
& \frac34\|\phi  DV_{A,\ez}(Du) \|^2_{L^2(\Omega)} + \|\phi U_{A,\ez}(Du) \|^2_{L^2(\Omega)}\nonumber\\
&\le C\|\phi f\|^2_{L^2(\Omega)} +  \Big[C_2+\frac{C_3}{4C_\sharp} \Big] \||D\phi| V_{A,\ez}(Du)\|^2_{L^2(\Omega)} + C   \|\phi  V_{A,\ez}(Du) \|^2_{L^1(\Omega)}.
\end{align}

Next let $\{B_{r_k}\}_{1\le k\le  N}$ be a  covering of $\overline{\Omega}$ by balls $B_{r_k}$, with $r_\sharp/4\le r_k\leq r_\sharp$,
such that either $B_{r_k}$ is center on $\partial\Omega$, or $B_{r_k} \Subset\Omega$.
Note that the covering can be chosen in the way that the multiplicity $N$ of overlapping among the balls $B_{r_k}$ depends only on $n$.
Let $\{\phi_k\}_{k\in N}$ be a family of functions such that $\phi_{k}\in C^\infty_c(B_{r_k})$ and $|\nabla \phi_k|\le C_4(r_\sharp)^{-1}$, and that
$\{\phi_k^2\}_{1\le k\le  N}$ is a partition of unity associated with the covering $\{B_{r_k}\}_{k\in N}$.
Thus $\sum_{k\in N}\phi_k^2=1$ in $\overline{\Omega}$. By applying inequality \eqref{g2} with $\phi=\phi_k$ for each $k$, and summing the resulting inequalities, one obtains
\begin{align}\label{g2a}
& \frac34\|  DV_{A,\ez}(Du) \|^2_{L^2(\Omega)} + \| U_{A,\ez}(Du) \|^2_{L^2(\Omega)}\nonumber\\
&\le C\| f\|^2_{L^2(\Omega)} +  N C_4^2\Big[C_2+ \frac{C_3}{4C_\sharp}\Big ] \frac1{r_\sharp^2}\|V_{A,\ez}(Du)\|^2_{L^2(\Omega)} + C  \|  V_{A,\ez}(Du) \|^2_{L^1(\Omega)}.
\end{align}
Choose $\eta$ small enough so that $$N C_4^2\Big[C_2+\frac{C_3}{4C_\sharp} \Big] \frac1{r_\sharp^2}\eta\le \frac14.$$
According to Lemma \ref{cmineq2},  we have
\begin{align}\label{g2a}
& \frac12\|  DV_{A,\ez}(Du) \|^2_{L^2(\Omega)} + \| U_{A,\ez}(Du) \|^2_{L^2(\Omega)} \le C\| f\|^2_{L^2(\Omega)}  + C  \|  V_{A,\ez}(Du) \|^2_{L^1(\Omega)}.
\end{align}
By $ \|V_\ez(AD)\|_{L^1(\Omega)}^2\le  C\|f\|_{L^2(\Omega)}^2$, we obtain the desired \eqref{end-reg-es0-1}.
 \end{proof}


\subsection{ Proof of Theorem \ref{thm:smooth}}

\begin{lem}\label{lem2.8}
\begin{align}\label{e3.x3-gen3}
& \frac34 \| DV_{A,\ez}(Du) \|^2_{L^2(\Omega)} + \|U_{A,\ez}(Du) \|^2_{L^2(\Omega)}\le C \| f\|^2_{L^2(\Omega)} +C _1\||DA|V_{A,\ez}(Du)\|^2_{L^2(\Omega)} .
\end{align}
\end{lem}
\begin{proof}
Applying Lemma \ref{L2} and Lemma \ref{boundaryterm}(ii), we have
\begin{align}\label{e3.x3-gen}
& \|\phi  DV_{A,\ez}(Du) \|^2_{L^2(\Omega)} + \|\phi U_{A,\ez}(Du) \|^2_{L^2(\Omega)}\nonumber\\
&\le C\|\phi f\|^2_{L^2(\Omega)} +C _1\||DA|\phi V_{A,\ez}(Du)\|^2_{L^2(\Omega)}\nonumber \\
&\quad +[C_2+C_\ast\Psi_\cb(r)] \||D\phi| V_{A,\ez}(Du)\|^2_{L^2(\Omega)}
+ C_\ast\Psi_\cb(r) \|\phi V_{A,\ez}(Du)\|^2_{L^2(\Omega)}
\end{align}

Next let $\{B_{r_k}\}_{1\le k\le  N}$ be a  covering of $\overline{\Omega}$ by balls $B_{r_k}$, with $r_\ast/4\le r_k\leq r_\ast$,
such that either $B_{r_k}$ is center on $\partial\Omega$, or $B_{r_k} \Subset\Omega$.
Note that the covering can be chosen in the way that the multiplicity $N$ of overlapping among the balls $B_{r_k}$ depends only on $n$.
Let $\{\psi_k\}_{k\in N}$ be a family of functions such that $\psi_{k}\in C^\infty_c(B_{r_k})$ and $|D\psi_k|\le C_4(r_*)^{-1}$, and that
$\{\psi_k^2\}_{1\le k\le  N}$ is a partition of unity associated with the covering $\{B_{r_k}\}_{k\in N}$.  By applying inequality \eqref{g2} with $\phi=\psi_k$ for each $k$, and summing the resulting inequalities, one obtains
\begin{align*}
& \| DV_{A,\ez}(Du) \|^2_{L^2(\Omega)} + \|U_{A,\ez}(Du) \|^2_{L^2(\Omega)}\nonumber\\
&\le C\| f\|^2_{L^2(\Omega)} +C _1\||DA|V_{A,\ez}(Du)\|^2_{L^2(\Omega)}\nonumber \\
&\quad +NC_4^2[C_2+C_\ast\Psi_\cb(r_\ast)] \frac1{r_\ast^2}\| V_{A,\ez}(Du)\|^2_{L^2(\Omega)}
+ C_\ast\Psi_\cb(r_\ast) \| V_{A,\ez}(Du)\|^2_{L^2(\Omega)}
\end{align*}
Choose $\eta>0$ small enough such that
$$\Big\{NC_4^2[C_2+C_\ast\Psi_\cb(r_\ast)] \frac1{r_\ast^2}+ C_\ast\Psi_\cb(r_\ast)\Big\}\eta\le \frac14.$$
Applying Lemma 2.6 we have
\begin{align*}
& \frac34 \| DV_{A,\ez}(Du) \|^2_{L^2(\Omega)} + \|U_{A,\ez}(Du) \|^2_{L^2(\Omega)}\nonumber\\
&\le C \| f\|^2_{L^2(\Omega)} +C _1\||DA|V_{A,\ez}(Du)\|^2_{L^2(\Omega)}  + C\| V_{A,\ez}(Du)\|^2_{L^1(\Omega)}
\end{align*}
By Corollary  \ref{L1-esti} we further have \eqref{e3.x3-gen} as desired.
\end{proof}

\begin{proof}[Proof of Theorem \ref{thm:smooth}(i)]
Since \eqref{e3.x3-gen3} is similar to \eqref{e3.x3-con1},  by exactly the same argument as the proof of
Theorem \ref{thm:convex-smooth} (i), we obtain Theorem \ref{thm:smooth}(i).
\end{proof}

\begin{proof}[Proof of Theorem \ref{thm:smooth}(ii)]
It suffices to show that
\begin{align}\label{yyy1} C _1\||DA|V_{A,\ez}(Du)\|^2_{L^2(\Omega)} \le  \frac12 \| D  V_{A,\ez}(Du)  \|^2_{L^2(\Omega)} +  C\|f\|_{L^2(\Omega)}^2
  \end{align}
To this end, let $\{B_{r_k}\}_{1\le k\le  N}$ be a  covering of $\overline{\Omega}$ as in the  proof of Theorem \ref{thm:convex-smooth} (ii)
and    correspondingly
$\{\phi_k\}_{k\in N}$ be therein.  Write
\begin{align*}
C _1\||DA|V_{A,\ez}(Du)\|^2_{L^2(\Omega)} & = \sum_k C _1\||DA|\phi_k V_{A,\ez}(Du)\|^2_{L^2(\Omega)}.
  \end{align*}
Note that  $C _1\||DA|\phi_k V_{A,\ez}(Du)\|^2_{L^2(\Omega)}$ is bounded by \eqref{xxx1} with $\phi=\phi_k$.
Thus
 \begin{align}\label{xxx2}
 & C _1\||DA|V_{A,\ez}(Du)\|^2_{L^2(\Omega)} \nonumber\\
  &\le  \frac14 \sum_k \|\phi_k  D  V_{A,\ez}(Du)  \|^2_{L^2(\Omega)}
 +     \frac{C_3}{4C_\sharp} \Big[
 \sum_k  \| | D\phi_k    |V_{A,\ez}(Du) \|^2_{L^2(\Omega)} + C  \sum_k  \|\phi _k V_{A,\ez}(Du) \|^2_{L^1(\Omega)}\Big]\nonumber\\
&\le  \frac14 \| D  V_{A,\ez}(Du)  \|^2_{L^2(\Omega)} +     \frac{C_3}{4C_\sharp  } \Big[  \frac{NC_4^2}{  r_\sharp^2}
  \|   V_{A,\ez}(Du) \|^2_{L^2(\Omega)} + C   \| V_{A,\ez}(Du) \|^2_{L^1(\Omega)}\Big].
  \end{align}
Choose $\eta$ small enough so that
 $$\frac{C_3}{4C_\sharp  }   \frac{NC_4^2}{  r_\sharp^2}\eta\le \frac14.$$
Applying Lemma \ref{cmineq2},  we have
$$ C _1\||DA|V_{A,\ez}(Du)\|^2_{L^2(\Omega)} \le  \frac12 \| D  V_{A,\ez}(Du)  \|^2_{L^2(\Omega)} +    C   \| V_{A,\ez}(Du) \|^2_{L^1(\Omega)}.
$$
By $ \|V_\ez(AD)\|_{L^1(\Omega)}^2\le  C\|f\|_{L^2(\Omega)}^2$, one has  \eqref{yyy1} as desired.
%
\end{proof}

\section{Proof of key Lemma  \ref{L21}}\label{S2}

 In this section, we  always let
  $1<p<\fz$ and $\ez\in(0,1]$, and
suppose  that  $\Omega$ is a smooth domain and  $A\in \mathcal E_L(\Omega)\cap C^\fz(\Omega)$.
To get Lemma \ref{L21}, it is suffices to prove the following two lemmas.

\begin{lem}\label{lem:key-ineq}
For any   $u\in C^3(\Omega)$,  we have
\begin{align}\label{div-inequal-1}
 (\mathcal L_{\ez,A,p}u)^2&\ge \frac34\min\{1,(p-1)^2\} [U_{A,\ez}(Du)]^2
- C|DA|^2|V_{A,\ez}(Du)|^2\nonumber\\
& \quad
+   \mathrm{div}\,\big\{(|\sqrt{A}Du|^2+\epsilon)^{p-2}[\mathrm{div}\,(ADu)ADu-(ADu\cdot D )ADu]\big\}\  \mbox{in $\Omega$},
\end{align}
where  $C >1$ is  a constant depending on $n,p,L$.
\end{lem}

\begin{lem}\label{lem:key-ineq-2} For any   $u\in C^3(\Omega)$ and any $\eta\in(0,1)$,  we have
\begin{align}\label{div-inequal-2}
& \left|\mathrm{div}\,\big\{(|\sqrt{A}Du|^2+\epsilon)^{p-2}[\mathrm{div}\,(ADu)ADu-(ADu\cdot D )ADu]\big\}\right. \nonumber\\
&\quad\quad\quad\quad\quad
-\left.
 \mathrm{div}\,\big\{(|\sqrt{A}Du|^2+\epsilon)^{p-2}[\mathrm{tr}\,(AD^2u)ADu - AD^2u ADu]\big\}\right| \nonumber\\
&\le \eta [U_{A,\ez}(Du)]^2
 + \frac C\eta |DA|^2|V_{A,\ez}(Du)| ^2\quad\mbox{in $\Omega$},
\end{align}
where  $C >1$ is  a constant depending on $n,p,L$.
\end{lem}

\begin{proof}[Proof  of Lemma \ref{L21}] The inequality \eqref{div-inequal} with ${\bf I}$  given by \eqref{I2} follows from \eqref{div-inequal-1}.
Moreover   the inequality \eqref{div-inequal} with ${\bf I}$  given by \eqref{I1} follows from \eqref{div-inequal-1} and   \eqref{div-inequal-2} with  $\eta=\frac1{4}\min\{1,(p-1)^2\}$.
\end{proof}



To prove Lemmas \ref{lem:key-ineq} and \ref{lem:key-ineq-2} we need the following two lemmas.
\begin{lem}\label{lem3.3}
For any symmetric $n\times n$ matrix $M$ and any  vector $\xi\in\rn$ with $|\xi|\le 1$, we have
\begin{equation}\label{matrixin}|M|^2\ge 2|M\xi|^2-|M\xi\cdot\xi|^2.
\end{equation}
\end{lem}

\begin{proof}
If  $|\xi|=1$, then $\xi= Oe_n$ for some orthogonal matrix $O$.
 Write $O^TMO=(m_{ij})_{1\le i,j\le n}$.
We have $$ |M|^2=|O^TMO|^2=\sum_{1\le i,j\le n}m_{ij}^2\ge \sum_{i=1}^n m_{in}^2 +\sum_{i=1}^n m_{ni}^2-m_{nn}^2.$$
Then \eqref{matrixin} follows from
  $$m_{nn} =O^TMOe_n\cdot e_n=MOe_n\cdot Oe_n= M\xi\cdot\xi,$$
and
$$\sum_{i=1}^n m_{ni}^2=\sum_{i=1}^n m_{in}^2 =|O^TMOe_n|^2=|M\xi|^2,$$
where we note $O^TMOe_n=(m_{in})_{1\le i\le n}$ and use symmetry.

Below assume $|\xi|<1 $. If  $|\xi|=0$ and \eqref{matrixin} holds trivially.
It $0<|\xi|<1$, applying \eqref{matrixin} to $\xi/|\xi|$, we have
$$|M|^2\ge 2\frac{|M\xi|^2}{|\xi|^2}-\frac{|M\xi\cdot\xi|^2}{|\xi|^4}=    \frac{|M\xi|^2}{|\xi|^2}+
 \frac{ |M\xi|^2|\xi|^2- |M\xi\cdot\xi|^2}{|\xi|^4}.$$
Since the Cauchy-Schwartz inequality gives   $|M\xi|^2|\xi|^2\ge |M\xi\cdot\xi|^2$, by $|\xi|<1$ and
$|\xi|+ 1/|\xi|\ge 2$ we have
$$|M|^2\ge    \frac{|M\xi|^2}{|\xi|^2}+
|M\xi|^2|\xi|^2- |M\xi\cdot\xi|^2 \ge 2 |M\xi|^2-|M\xi\cdot\xi|^2.$$

\end{proof}

\begin{rem}\rm
After diagonalizing   $M $, the inequality \eqref{matrixin}  reads as
  \begin{align}\label{I2-eq12}
 \left(\sum^n_{i=1}\xi_i^2\lambda_i\right)^2-2\sum^n_{i=1}\xi_i^2\lambda_i^2+\sum^n_{i=1}\lambda_i^2\geq 0,\quad  \quad\mbox{$\forall\xi\in\rn$ with $|\xi|\le 1$},
\end{align}
which was proved by \cite[Lemma  3.2]{CMaz} and used to prove Lemma \ref{lem:key-ineq} with $A=I_n$.
 Above we give a much simpler proof to \eqref{I2-eq12}.
 \end{rem}

\begin{lem}\label{lem3.2-a}
Assume that  $M $ is an $n\times n$ symmetric   matrix with $\frac1 L\le M\xi\cdot\xi\le L$ for all $\xi\in\rn$ with $|\xi|=1$.  Then
$\frac1 {\sqrt L}\le \sqrt M\xi\cdot\xi\le \sqrt L$ for all $\xi\in\rn$ with $|\xi|=1$. Moreover, for any $n\times n$ matrix  $H$,  we have
\begin{equation}\label{equv0}
\frac{1}{\sqrt{L }}|H|\leq|\sqrt{M}H|\leq \sqrt{L } |H| \quad\textrm{and}\quad \frac{1}{\sqrt L }|H|\leq|H\sqrt{M}|\leq \sqrt{L } |H|.
\end{equation}
\end{lem}

\begin{proof}[Proof of Lemma \ref{lem3.2-a}]
We can find an orthogonal matrix $O$ such that
 $O^TMO={\rm diag}\{\lz_1,\cdots,\lz_n\}$ with $\frac1L\le \lz_1\le\cdots\le \lz_n\le L$
and $O^T\sqrt {M}O={\rm diag}\{\sqrt {\lz_1},\cdots,\sqrt {\lz_n}\}$ with $\frac1{\sqrt L}\le \lz_1\le\cdots\le \lz_n\le \sqrt {L}$. Thus
$\frac1 {\sqrt L}\le \sqrt M\xi\cdot\xi\le \sqrt L$ for all $\xi\in\rn$ with $|\xi|=1$.

 Given any matrix $H$, we have
$$ |\sqrt{M}H|^2 =   |O^T\sqrt{M}OO^TH|^2 =|\mathrm{diag}\Big(\lambda_1 ,\cdots,\lambda_n \Big)O^TH|,$$
and hence
$$ \frac1{  L}|H|^2\le \lz_1 |O^TH|^2\le |\sqrt{M}H|^2\le \lz_n |O^TH|^2\le L|H|^2.$$
 The same  result holds for $|H\sqrt M|^2$.

\end{proof}

Below we prove Lemma \ref{lem:key-ineq} and Lemma \ref{lem:key-ineq-2}.
\begin{proof}[Proof of Lemma \ref{lem:key-ineq}]


For simple, we write
$$\Pi:=[\mathcal L_{\ez,A,p}u]^2-\mathrm{div}\,\big\{(|\sqrt{A}Du|^2+\epsilon)^{p-2}[\mathrm{div}\,(ADu)ADu-(ADu\cdot D )ADu]\big\},$$
and then,   \eqref{div-inequal-1} is equivalent to
\begin{align}\label{L-D}
\Pi
&\ge  \frac34\min\{1,(p-1)^2\}(| D_Au|^2+\epsilon)^{p-2}|   D^2_Au  |^2 - C(n,p,L)|DA|^2|V_{A,\ez}(Du)|^2.
\end{align}


We prove \eqref{L-D} as below. For vectors
 $a,b\in \rr^n$, we  use $\langle a,b\rangle=a\cdot b$ to denote  the usual inner product.
We always use Einstein summation convention, that is, $a_ib_i=a^ib_i=\sum_{i=1}^na_ib_i$.
 For short,
we always write
$Dv=(\partial_i v)_{1\le i\le n}=(v_i)_{1\le i\le n}$;
write \begin{equation}
  D_A v := \sqrt{A}Du,\quad \quad D^2_A v:= \sqrt{A} D^2 v \sqrt{A};
\end{equation}
  also write $$\xi\cdot Dv=\xi_i\partial_iv , \quad \langle (DA)\xi,\xi\rangle =\xi^T(DA)\xi =(a_k^{ij}\xi_i\xi_j)_{1\le k\le n}.$$

 Firstly, a direct calculation yields
\begin{align*}
  \mathcal L_{\ez,A,p}u&=\mathrm{div}\,\left((| D_Au|^2+\epsilon)^{\frac{p-2}{2}}ADu\right)\\
&=(| D_Au|^2+\varepsilon)^{\frac{p-2}{2}}\mathrm{div}\,\left(ADu\right) +(p-2)(| D_Au|^2+\epsilon)^{\frac{p-4}{2}}\Big\langle D\frac{| D_Au|^2}{2},ADu\Big\rangle,
\end{align*}
and hence
\begin{align}\label{LezAp2}
 [\mathcal L_{\ez,A,p}u]^2& =(|D_Au|^2+\epsilon)^{p-2}[\mathrm{div}\,\left(ADu\right)]^2\nonumber\\
 &\quad+2(p-2)(|D_Au|^2+\epsilon)^{p-3}\mathrm{div}\,(ADu)\Big\langle D\frac{|D_Au|^2}{2},ADu\Big\rangle\nonumber\\
  &\quad+(p-2)^2(|D_Au|^2+\epsilon)^{p-4}\Big\langle D\frac{|D_Au|^2}{2},ADu\Big\rangle^2.
\end{align}

We also observe that
\begin{align}\label{equa1}
  &\mathrm{div}\,\big\{(| D_Au|^2+\varepsilon)^{p-2}[\mathrm{div}\,(ADu)ADu-(ADu\cdot D )ADu]\big\} \nonumber \\
  &= (| D_Au|^2+\varepsilon)^{p-2}\mathrm{div}\,[\mathrm{div}\,(ADu)ADu-(ADu\cdot D )ADu]\nonumber \\
  &\quad +2 (p-2)(| D_Au|^2+\varepsilon)^{p-3}\mathrm{div}\,(ADu)\,\Big\langle D \frac{| D_Au|^2}{2} ,
   ADu \Big\rangle \nonumber \\
  &\quad - 2 (p-2)(| D_Au|^2+\varepsilon)^{p-3}\,\Big\langle D \frac{| D_Au|^2}{2},
    (ADu\cdot D )ADu \Big\rangle.
\end{align}
Write
\begin{align}\label{divdiv}
  &\mathrm{div}\,[\mathrm{div}\,(ADu)ADu-(ADu\cdot D )ADu] \nonumber\\ &
  = [\mathrm{div}\, (ADu)]^2 + (ADu\cdot D) \mathrm{div}\,(ADu) -  (ADu\cdot D ) \mathrm{div}\,(ADu) -  \big(\partial_i(a^{js} u_s)\big) \cdot \big(\partial_j (a^{ik} u_k)\big)\nonumber\\
 & = [\mathrm{div}\, (ADu)]^2-  \big(\partial_i(a^{js} u_s)\big) \cdot \big(\partial_j (a^{ik} u_k)\big).
 \end{align}
Since $$
-  \big(\partial_i(a^{js} u_s)\big) \cdot \big(\partial_j (a^{ik} u_k)\big)=
 - a^{js}_{i}u_s a^{ik}_j u_k-2a^{js}u_{is}a^{ik}_{j}u_k-a^{ik}u_{jk}a^{js}u_{is},$$
and
\begin{align*}
u_{jk}a^{js}u_{is}a^{ik}&=\mathrm{tr}\,(A D^2 u A D^2 u )\nonumber\\
&=\mathrm{tr}\,(\sqrt{A} D^2 u A D^2 u \sqrt{A})= \mathrm{tr}\, (\sqrt{A}D^2 u \sqrt{A})^2=\mathrm{tr}\,(D_A^2 u )^2=|D_A^2 u |^2,
\end{align*}
we have
\begin{align}\label{divdiv}
  &\mathrm{div}\,[\mathrm{div}\,(ADu)ADu-(ADu\cdot D )ADu] \nonumber\\
  & = [\mathrm{div}\, (ADu)]^2-  |D_A^2 u |^2- a^{js}_{i}u_s a^{ik}_j u_k-2a^{js}u_{is}a^{ik}_{j}u_k.
 \end{align}

From \eqref{LezAp2}, \eqref{equa1} and \eqref{divdiv}, we deduce that
\begin{align}\label{equa2}
\Pi&= (| D_Au|^2+\varepsilon)^{p-2} |D_A^2 u |^2\nonumber \\
&\quad+2(p-2) (| D_Au|^2+\varepsilon)^{p-3} \Big\langle D\frac{| D_Au|^2}{2}, (ADu\cdot D )ADu \Big\rangle \nonumber \\
&\quad+
(p-2)^2 (| D_Au|^2+\varepsilon)^{p-4} \Big\langle D\frac{| D_Au|^2}{2},ADu\Big\rangle^2 \nonumber \\
&\quad+(| D_Au|^2+\varepsilon)^{p-2}[ a^{js}_{i}u_s a^{ik}_j u_k+2a^{js}u_{is}a^{ik}_{j}u_k]\nonumber\\
 &=: \Pi_1+\Pi_2+\Pi_3+\Pi_4.
\end{align}


Next we bound $\Pi_2$, $\Pi_3$ and $\Pi_4$   as below.
To bound $\Pi_2$, write
\begin{equation}\label{Df1}
 D\frac{|D_Au|^2}{2}= \frac12 D\langle ADu,Du\rangle = D^2uADu +
 \frac12 (Du)^T (DA)Du. \end{equation}
Then
\begin{align*}
\langle D\frac{|D_Au|^2}{2}, ADu\rangle
& =\langle D^2uADu,ADu\rangle + \frac12\langle  (Du)^T (DA)Du,ADu\rangle\\
&=
(D_A u)^T D^2_A u  D_A u+ \frac12 (Du)^T [(ADu\cdot D)A]Du,
\end{align*}
that is,
\begin{align*}
 \Big\langle D\frac{|\sqrt{A}Du|^2}{2}, ADu\Big\rangle^2 & =[(D_A u)^T D^2_A u   D_A u]^2+ \frac14 [  (Du)^T [(ADu\cdot D)A]Du]^2 \nonumber \\
&\quad +  [(D_Au)^T D^2_A u D_A u]  (Du)^T [(ADu\cdot D)A]Du. \nonumber
\end{align*}
Multiplying both side by $(p-2)^2$ and using Cauchy-Schwartz inequality, we obtain
\begin{align}\label{DA3}
  &(p-2)^2\Big\langle D\frac{|\sqrt{A}Du|^2}{2}, ADu\Big\rangle^2\nonumber\\
  &\quad\ge   (p-2)^2[(D_A u)^T D^2_A u   D_A u]^2 -\eta |D_A^2u|^2 |D_Au|^4-
\frac{C}\eta |D A|^2|D_Au|^4|Du|^2.
\end{align}
Multiplying both side by $ (|D_Au|^2+\ez)^{p-4}$,   we obtain
\begin{align}\label{DA3}
 \Pi_2 & \ge   (p-2)^2(|D_Au|^2+\ez)^{p-4} [(D_A u)^T D^2_A u   D_A u]^2 -\eta \Pi_1-
\frac{C}\eta |D A|^2 |V_{A,\ez}(Du)|^2.
\end{align}

To bound $\Pi_3$,  by \eqref{Df1} and
$$
(ADu\cdot D )ADu= AD^2u  ADu+[(ADu\cdot D)A] Du,$$
we have
\begin{align*}
& \Big\langle D \frac{|\sqrt{A}Du|^2}{2}, \,(ADu\cdot D )ADu \Big\rangle\\
 &=\Big\langle  D^2uADu + \frac12 (Du)^T (DA)Du, AD^2u  ADu+[(ADu\cdot D)A] Du\Big\rangle\\
&= \langle  D^2u  A D u  ,AD^2u  ADu\rangle+
\langle  D^2u   A D u,[(ADu\cdot D)A] Du \rangle\\
&\quad+  \frac12\langle    (Du)^T (DA)Du, AD^2u  ADu  \rangle+
\frac12\langle    (Du)^T (DA)Du,  [(ADu\cdot D)A] Du \rangle.
 \end{align*}
By the Cauchy-Schwartz inequality one gets
\begin{align}\nonumber\label{DA4}
& 2(p-2)\Big\langle D \frac{|\sqrt{A}Du|^2}{2}, \,(ADu\cdot D )ADu \Big\rangle\nonumber\\
&\ge 2(p-2)  |D_A2uD_Au|^2-  \eta |D_A^2u|^2 |D_Au|^2-
\frac{C}\eta |D  A|^2|D_Au|^2|Du|^2.
\end{align}
Multiplying both side by $ (|D_Au|^2+\ez)^{p-3}$   we obtain
\begin{align}\label{DA3}
 \Pi_3 & \ge   (p-2) (|D_Au|^2+\ez)^{p-3} |D^2_A u   D_A u|^2 -\eta \Pi_1-
\frac{C}\eta |D A|^2 |V_{A,\ez}(Du)|^2.
\end{align}

For $\Pi_4$, we observe that
$$ a^{js}_{i}u_s a^{ik}_j u_k\ge -|D A|^2|Du|^2.$$
By Lemma \ref{lem3.2-a} and the Cauchy-Schwartz inequality,  one also has
$$2a^{js}u_{is}a^{ik}_{j}u_k \ge -|AD^2u||D A||Du|\ge -L|D_A^2u||DA||Du|\ge
-\eta| D_A^2 u|^2-\frac C\eta |DA|^2|Du|^2.
$$
Thus
 $$\Pi_4\ge -\eta \Pi_1-\frac C\eta |DA|^2|V_{A,\ez}(Du)|^2.$$

From the lower bound of $\Pi_2, \Pi_3$ and $\Pi_4$,
 noting $(| D_Au|^2+\epsilon)^{p-2}  | Du |^2\le L^2
|V_{A,\ez}(Du)|^2$,
we attain
\begin{align}\label{equa4}
\Pi&\ge  (1-3\eta) (| D_Au|^2+\epsilon)^{p-2} |D_A^2u|^2+2(p-2) (| D_Au|^2+\epsilon)^{p-3}  |D_A^2uD_Au|^2 \nonumber \\
&\quad+
(p-2)^2 (| D_Au|^2+\epsilon)^{p-4}   |(D_Au)^TD_A^2uD_Au|^2-\frac C\eta |DA|^2|V_{A,\ez}(Du)|^2.
\end{align}

 If $p\ge 2$, taking $\eta =\frac1{16}$ in \eqref{equa4}, noting that the second and third terms in the right hand side of \eqref{equa4} are  nonnegative,  we have  \eqref{L-D}.

Below we assume that $ 1<p<2$.   Then  $1-p(2-p)=(p-1)^2>0$ and hence,  $1>p(2-p)>0$.
We split the coefficient $(1-3\eta)$ of  $(| D_Au|^2+\varepsilon)^{p-2} |D_A^2u|^2$ in  the right hand side of \eqref{equa4}   as
\begin{align} \label{Inst1}(1-3\eta) &=  [1-p(2-p)-3\eta]
  + p(2-p)  \ge \frac 34  (p-1)^2
 + p(2-p)
\end{align}
where we choose  $\eta= \frac1{16}(p-1)^2$.   Then
\begin{align}\label{equa5}
\Pi&\ge \frac34(p-1)^2  (| D_Au|^2+\varepsilon)^{p-2} |D_A^2u|^2-\frac C\eta |DA|^2  |V_{A,\ez}(Du)|^2\nonumber \\
&\quad+(2-p)(| D_Au|^2+\varepsilon)^{p-2} \Big\{p|D^2_Au|^2-2\frac{|D^2_Au D_Au|^2}{|D_Au|^2+\ez}+ (2-p)
\frac{|(D_Au)^TD^2_Au D_Au|^2}{(|D_Au|^2+\ez)^2}  \Big\}.
\end{align}
Applying \eqref{matrixin} to $D^2_Au$ and $D_Au/\sqrt{|D_Au|^2+\ez}$, and multiplying both side  $p$,
 one has
\begin{align} \label{Inst2}p |D^2_Au|^2&\ge 2p \frac{|D^2_Au D_Au|^2}{|D_Au|^2+\ez} -p
\frac{|(D_Au)^TD^2_Au D_Au|^2}{(|D_Au|^2+\ez)^2}\nonumber\\
&=2  \frac{|D^2_Au D_Au|^2}{|D_Au|^2+\ez}+ \Big[(2p-2) \frac{|D^2_Au D_Au|^2}{|D_Au|^2+\ez}-p
\frac{|(D_Au)^TD^2_Au D_Au|^2}{(|D_Au|^2+\ez)^2}\Big]
.\end{align}
Since
  Cauchy-Schwartz inequality gives
$$
 |D^2_Au D_Au|^2  \ge
\frac{|(D_Au)^TD^2_Au D_Au|^2}{ |D_Au|^2+\ez} ,$$
we obtain
\begin{equation} \label{Inst3} p |D^2_Au|^2\ge 2\frac{|D^2_Au D_Au|^2}{|D_Au|^2+\ez}- (2-p)
\frac{|(D_Au)^TD^2_Au D_Au|^2}{(|D_Au|^2+\ez)^2}.
\end{equation}
This inequality shows us the third term in the right hand side of \eqref{equa5} is nonnegative.
Hence,  \eqref{L-D}  follows from \eqref{equa5}.
 \end{proof}

\begin{rem} \rm In the case $A=I_n$, we can take $\eta=0$ in the above proof to get
\begin{align}\label{equax}
\Pi &=[{\rm div}(|Du|^2+\ez)^{\frac{p-2}2}Du]^2-{\rm div}(|Du|^2+\ez)^{p-2}[\Delta u Du-D^2uDu] \nonumber\\
&\ge \min\{1,(p-1)^2\} (|Du|^2+\ez)^{ {p-2} }|D^2u|^2.
\end{align}
Indeed, when $A=I_n$,  \eqref{equa2} becomes
\begin{align}\label{equax}
\Pi
&= (| D u|^2+\epsilon)^{p-2} \Big[|D^2 u |^2 +2(p-2)   \frac{|D^2uDu|^2}{|Du|^2+\ez}
+ (p-2)^2   \frac{|(Du)^TD^2uDu|^2}{(|Du|^2+\ez)^2}\Big ].
\end{align}
If $p\ge2$, one has $\Pi \ge (| D u|^2+\epsilon)^{p-2}|D^2 u |^2$.
If $1<p<2$,    by splitting the coefficient 1 of $|D^2 u |^2$ as $(p-1)^2 + p(2-p) $ we have
\begin{align*}
\Pi&= (p-1)^2(| D u|^2+\epsilon)^{p-2} |D^2 u |^2 \\
&\quad+(2-p)(| D u|^2+\epsilon)^{p-2}[ p|D^2u|^2-2   \frac{|D^2uDu|^2}{|Du|^2+\ez}
+ (2-p)   \frac{|(Du)^TD^2uDu|^2}{(|Du|^2+\ez)^2} ].
\end{align*}
Applying Lemma \ref{lem3.3} to $D^2u$ and $Du/ \sqrt{ |Du|^2+\ez }$, by an argument same as \eqref{Inst3} we  obtain
$$p |D^2 u|^2\ge 2\frac{|D^2 u D u|^2}{|D u|^2+\ez}- (2-p)
\frac{|(D u)^TD^2 u D u|^2}{(|D u|^2+\ez)^2}.$$
Hence   $
\Pi \ge (p-1)^2(| D u|^2+\epsilon)^{p-2} |D^2 u |^2 $ as desired.
\end{rem}

\begin{proof}[Proof of Lemma \ref{lem:key-ineq-2}] A direct calculation gives
\begin{align*}
 &\mathrm{div}\,\big\{(|\sqrt{A}Du|^2+\epsilon)^{p-2}[\mathrm{div}\,(ADu)ADu-(ADu\cdot D )ADu]\big\}\\
&\quad\quad\quad-\mathrm{div}\,\big\{(|\sqrt{A}Du|^2+\epsilon)^{p-2}[\mathrm{tr}(AD^2u)ADu- AD^2uADu ]\big\}\\
&\quad=(|\sqrt{A}Du|^2+\epsilon)^{p-2}\\
&\quad\quad\quad\times \mathrm{div}\left\{[\mathrm{div}\,(ADu)ADu-(ADu\cdot D )ADu] -   [\mathrm{tr}(AD^2u)ADu- AD^2uADu ]\right\} \\
&\quad\quad+   D(|\sqrt{A}Du|^2+\epsilon)^{p-2} \\
&\quad\quad\quad\cdot  \left\{[\mathrm{div}\,(ADu)ADu-(ADu\cdot D )ADu]-[\mathrm{tr}(AD^2u)ADu- AD^2uADu ]\right\}\\
&\quad=:J_1+J_2.
\end{align*}

It then suffices to prove that, for $i=2$,
$$ J_i\ge   -\eta [U_{A,\ez}(Du)]^2-\frac C\eta|DA|^2|V_{ A,\ez}(Du)|^2$$
A direct calculation also yields
\begin{align}\label{diff}&[\mathrm{div}\,(ADu)ADu-(ADu\cdot D )ADu]- [\mathrm{tr}(AD^2u)ADu- AD^2uADu ]\nonumber\\
&\quad=({\rm div} A \cdot Du) ADu - [(ADu\cdot D)A]Du,
\end{align}
where
 ${\rm div} A \cdot Du= a_j^{jk}u_k$.
Since
$$D(| D_Au|^2+\epsilon)^{p-2} =  (p-2) (| D_Au|^2+\epsilon)^{p-4}
[2D^2uADu+  (Du)^T (D A)Du],$$
by the Cauchy-Schwartz inequality and Lemma \ref{lem3.2-a} one has
\begin{align*}
J_2&= (p-2) (| D_Au|^2+\epsilon)^{p-4}
[2D^2uADu+  (Du)^T (D A)Du]\\
&\quad\quad\cdot \left\{({\rm div A}\cdot Du) ADu - [(ADu\cdot D)A]Du\right\}\\
&=(p-2) (| D_Au|^2+\epsilon)^{p-4} \{ 2({\rm div} A \cdot Du)(ADu)^TD^2uADu
- 2D^2uADu \cdot [(ADu\cdot D)A]Du \\
&\quad\quad + ({\rm div} A \cdot Du)  (Du)^T (D A)Du\cdot A Du-
(Du)^T (D A)Du\cdot  [(ADu\cdot D)A]Du
\}\\
&\ge \eta (| D_Au|^2+\epsilon)^{p-2} | D_A^2u| ^2
-\frac C\eta |DA|^2 (| D_Au|^2+\epsilon)^{p-2} |Du|^2\\
&\ge  -\eta [U_{A,\ez}(Du)]^2 -\frac C\eta |DA|^2|V_{ A,\ez}(Du)|^2,
\end{align*}
for any $\eta\in(0,1)$.
Moreover, by  the Cauchy-Schwartz inequality,  one  has
\begin{align*} &{\rm div}[({\rm div} A\cdot Du) ADu - [(ADu\cdot D)A]Du] \\
&=  \langle D^2u ({\rm div }A),   A Du\rangle+ ({\rm div} A \cdot Du)^2+({\rm div }A \cdot Du) {\rm tr} (AD^2u)\\
&\quad\quad -   a^{jk}_iu_k a_j^{il}u_l- a^{jk}u_{ki} a_j^{il}u_l-a^{jk}u_k a_j^{il}u_{li} \\
&\ge  -\eta| D_A^2u| ^2-\frac C\eta|DA|^2|\sqrt A Du|^2.
\end{align*}
Thus $$ J_1\ge   -\eta [U_{A,\ez}(Du)]^2-\frac C\eta|DA|^2|V_{ A,\ez}(Du)|^2$$
as desired.
\end{proof}

\section{Proof of key Lemma \ref{L2}}

In this section we prove  Lemma \ref{L2}.
We always let $1<p<\fz$ and $\ez\in(0,1]$ and
suppose that $\Omega, f$ and $A$  satisfy  assumptions (S1)-(S3).
 Let $u$ be a weak solution to
\eqref{app-ellip-syseps0} with  Dirichlet/Neumann $0$-boundary.

\begin{proof}[Proof of Lemma \ref{L2}]

Recall that \eqref{div-inequal}  gives
$$\frac12\min\{1,(p-1)^2\} [U_{A,\ez}(Du)]^2 \le  f^2 + C  |DA|^2|V_{A,\ez}(Du)|^2
- {\bf I}.$$
Since
\begin{align*}DV_{A,\ez}(Du)&=(p-2)(|D_Au|^2+\ez)^{\frac{p-4}2}[  D^2uADu\otimes ADu+\frac12(Du)^T(DA)Du\otimes ADu  ]\\
&\quad +(|D_Au|^2+\ez)^{\frac{p-2}2} [(DA)Du+AD^2u],
\end{align*}
by the Cauchy-Schwartz inequality one has
$$
|DV_{A,\ez}(Du)|^2\le L[1+|p-2|]|U_{A,\ez}(Du)|^2+ C|DA|^2|V_{A,\ez}(Du)|^2.$$
We then get
$$ [U_{A,\ez}(Du)]^2+ |DV_{A,\ez}(Du)|^2\le Cf^2 + C  |DA|^2|V_{A,\ez}(Du)|^2
-C{\bf I}.
$$
Multiplying both sides by $\phi^2$ for any $\phi\in C_c^\fz(\rn)$,  and integrating over $\Omega$,   we get
\begin{equation}
\begin{split}
&\int_{\Omega}\phi^2|DV_{A,\ez}(Du) |^2\mathrm{d}x + \int_{\Omega}\phi^2 |U_{A,\ez}(Du) |^2\mathrm{d}x \\
&\le  C \int_{\Omega}\phi^2f^2\mathrm{d}x+ C\int_{\Omega}\phi^2  |DA|^2|V_\ez(ADu| ^2\mathrm{d}x +\wz {\bf K}(\phi),
\end{split}
\end{equation}
where if ${\bf I }$ is given by \eqref{I1}, then
\begin{equation}\label{KK1}
\wz {\bf K}(\phi)=-C \int_{\Omega}\phi^2\mathrm{div}\,\big\{(|\sqrt{A}Du|^2+\epsilon)^{p-2}[\mathrm{tr}(AD^2u)ADu - AD^2uADu]\big\} \mathrm{d}x,
\end{equation}
and if ${\bf I }$ is given by \eqref{I2}, then
\begin{equation}\label{KK2}
\wz {\bf K}(\phi)=-
 C \int_{\Omega}\phi^2 \mathrm{div}\,\Big\{(|\sqrt{A}Du|^2+\epsilon)^{p-2}[\mathrm{div}\,(ADu)ADu-(ADu\cdot D )ADu]\Big\} \mathrm{d}x.
\end{equation}

Regards of $\wz {\bf K}(\phi)$ as in \eqref{KK1},
we use  the divergence theorem for \eqref{KK1} to get
 \begin{equation}
\begin{split}
 \wz {\bf K}(\phi)&=2C\int_{\Omega}\phi D\phi (|\sqrt{A}Du|^2+\epsilon)^{p-2}\,[\mathrm{tr}(AD^2u)ADu - AD^2uADu]\dd x\nonumber \\
&\quad-C\int_{\partial\Omega}\phi^2(|\sqrt{A}Du|^2+\epsilon)^{p-2}\Big[ \mathrm{tr}(AD^2u)ADu- AD^2uADu\Big] \cdot\nu\dd\mathcal{H}^{n-1}(x) \label{I1-4}.
\end{split}
 \end{equation}
 Owing to   Young's inequality,  one has
 \begin{align}\label{J-es2}
& 2C \int_{\Omega}\phi D\phi (|\sqrt{A}Du|^2+\epsilon)^{p-2}\,[\mathrm{tr}(AD^2u)ADu - AD^2uADu]\dd x\nonumber\\
 & \le   C \int_{\Omega} |\phi D\phi | (|\sqrt{A}Du|^2+\epsilon)^{p-2} | AD^2u |\, |ADu| \dd x \nonumber \\
   &\le  \frac14 \int_{\Omega}\phi^2(|\sqrt{A}Du|^2+\epsilon)^{p-2}|D^2_Au|^2\dd x + C \int_{\Omega} | D\phi |^2|A Du|^2  (|\sqrt{A}Du|^2+\epsilon)^{p-2}\dd x\nonumber\\
&= \frac14 \int_{\Omega} \phi^2[U_{A,\ez}(Du)]^2\,dx+C \int_{\Omega} | D\phi |^2|V_{A,\ez}(Du)|^2 \,dx.\nonumber
 \end{align}
We therefore obtain \eqref{e3.x3} with ${\bf K}(\phi)$ given by \eqref{K1}.

With regard to $\wz {\bf K}(\phi)$ as in \eqref{KK2}, we use the divergence theorem for \eqref{KK2} to get
 \begin{equation}
\begin{split}
 \wz {\bf K}(\phi)&= 2C\int_{\Omega}\phi D\phi (|\sqrt{A}Du|^2+\epsilon)^{p-2}\,[\mathrm{div}(ADu)ADu-(ADu\cdot D )ADu]\dd x \nonumber\\
&\quad-C\int_{\partial\Omega}\phi^2(|\sqrt{A}Du|^2+\epsilon)^{p-2}\Big[\mathrm{div}(ADu)ADu  -  (ADu\cdot D )ADu \Big]\cdot\nu \dd\mathcal{H}^{n-1}(x).
\end{split}
 \end{equation}
Owing to   Young's inequality,  one has
 \begin{align}\label{J-es}
& 2C \int_{\Omega}\phi D\phi (|\sqrt{A}Du|^2+\epsilon)^{p-2}[\mathrm{div}\,(ADu)ADu-(ADu\cdot D )ADu]\dd x\nonumber\\
 & \le   C \int_{\Omega} |\phi D\phi | (|\sqrt{A}Du|^2+\epsilon)^{p-2} |\nabla(ADu)|\, |ADu| \dd x \nonumber \\
   &\le  \frac14 \int_{\Omega}\phi^2(|\sqrt{A}Du|^2+\epsilon)^{p-2}|A D^2u|^2\dd x \nonumber \\
  & \quad + C \int_{\Omega}\big(|D\phi |^2|A Du|^2 + |\phi|^2 |DA|^2 | Du|^2\big)(|\sqrt{A}Du|^2+\epsilon)^{p-2}\dd x\nonumber\\
&\le  \frac14 \int_{\Omega} \phi^2[U_{A,\ez}(Du)]^2\,dx+C \int_{\Omega} |D \phi|^2|V_{A,\ez}(Du)|^2 \,dx+  C \int_{\Omega} |\phi|^2|DA|^2|V_{A,\ez}(Du)|^2 \,dx.
 \end{align}
We therefore obtain \eqref{e3.x3} with ${\bf K}(\phi)$ given by \eqref{K2}.

\end{proof}

\section{Proof of key Lemma \ref{boundaryterm}}

Given any bounded smooth domain  $ \Omega\subset\rn$, recall that   $\nu$ denotes the  unit outer normal  vector of the boundary $\partial\Omega$.
Below we write   $T(\partial\Omega)$  the tangential space of $\partial\Omega$, and let $ \{\tau_1,\cdots,\tau_{n-1}\}$ be  an normalized orthogonal basis of $ T(\partial \Omega)$ so that
   $\{\tau_1,\cdots,\tau_{n-1},\nu\}$ has the same orientation as $\{e_1,\cdots,e_n\}$.
 We use $\mathrm{div}_{T}$ and $\nabla_{T}$ to denote the divergence and the gradient operator on $\partial\Omega$.
We also recall that $\mathcal{B}$ is the second fundamental form of $\partial\Omega$, which is given by
$$\mathcal{B}(\xi,\eta )=-\frac{\partial \nu}{\partial\xi}\cdot \eta=-\sum_{k=1}^{n-1} (\xi\cdot\tau_k)  \frac{\partial \nu}{\partial\tau_k}\cdot\eta
=-\sum_{i,k=1}^{n-1}(\frac{\partial \nu}{\partial\tau_k}\cdot \tau_i) (\xi\cdot\tau_k)  (\eta\cdot \tau_i),  \quad\forall\xi,\eta\in T(\partial\Omega),$$
where $\frac{\partial }{\partial\xi} $ denote the derivative along the direction $\xi$.
We use $|\mathcal B|$ to denote  the norm of $\mathcal B$, that is,$$ |\mathcal B|=\sup_{|\xi|,|\eta|\le1}\mathcal B(\xi,\eta).$$
The  trace of $\mathcal B$ is given by
$$\mathrm{tr}\mathcal{B}=\sum_{i=1}^{n-1}\mathcal{B}(\tau_i,\tau_i)  =-\sum_{i=1}^{n-1} \frac{\partial \nu}{\partial\tau_i}\cdot \tau_i. $$

To prove Lemma \ref{boundaryterm}, we need a series of lemmas. Firstly, we need the following result to bound the term in ${\bf K}(\phi)$ as in \eqref{K1}. The proof of Lemma \ref{lem:3.3.1a} is postponed to Section 5.1.
\begin{lem}\label{lem:3.3.1a}
 Assume  that $u\in C^2(\overline \Omega) $ and $u=0$ on $\partial\Omega$.  We have
 \begin{align}\label{lem5.1-1}
   -[\mathrm{tr}(AD^2u)ADu- AD^2uADu]\cdot\nu  \le C |\mathcal B| |ADu|^2 \quad\mbox{on $\partial\Omega$}
 \end{align}
with the constant $C $ depending on $L$;
 moreover if $\Omega$ is convex, we have
 \begin{align}\label{lem5.1-2}
   -[\mathrm{tr}(AD^2u)ADu- AD^2uADu]\cdot\nu\le   0 \quad\mbox{on $\partial\Omega$}.
 \end{align}
\end{lem}

We also need the following result to  bound the term in ${\bf K}(\phi)$ as in \eqref{K2}. The proof of Lemma \ref{lem:3.3.1a} is postponed to Section 5.1.

\begin{lem} \label{l3.yy1} Assume that $u\in C^2(\overline\Omega)$ and $ADu\cdot \nu=0$ on $\partial\Omega$.
We have
$$-[\mathrm{div}(ADu)ADu - (ADu\cdot D )ADu]\cdot\nu  \le |\mathcal B| |ADu|^2 \quad\mbox{on $\partial\Omega$};
 $$
  moreover, if $\Omega$ is convex, we have
$$-[\mathrm{div}(ADu)ADu - (ADu\cdot D )ADu] \cdot\nu\le0\quad\mbox{on $\partial\Omega$} .$$
\end{lem}

The following trace inequality  was built up by Cianchi-Mazya \cite{CMaz} and will be used to prove Lemma \ref{boundaryterm} (ii).

\begin{lem}\label{CMineq} For any $x\in\partial \Omega$ and $0<r<1$, when $n\ge3$, one has
\begin{equation}\label{case1-1}
  \int_{\partial\Omega\cap B_r(x)}v^2|\mathcal{B}|\dd \mathcal{H}^{n-1}(y)\leq C\Psi_{\mathcal{B}}(r)\int_{\Omega\cap B_r(x)}|\nabla v|^2\dd y.
\end{equation}

\end{lem}
Now we are ready to prove Lemma \ref{boundaryterm}.

\begin{proof} [Proof of Lemma \ref{boundaryterm} (i).]  If $\phi$ is supported in $\Omega$, then ${\bf K}(\phi)=0$. Assume that the support of $\phi$ has non empty intersection with $\partial\Omega$.
Since $\Omega$ is convex,  when $u=0$ on $\partial\Omega$, applying Lemma \ref{lem:3.3.1a},  we have
$$-[\mathrm{tr}(AD^2u)ADu - AD^2uADu]\cdot\nu\le 0. $$
Hence, for the ${\bf K}(\phi)$ given by \eqref{K1}, we have
$$ {\bf K}(\phi)=- C
\int_{\partial\Omega}\phi^2(|\sqrt{A}Du|^2+\epsilon)^{p-2} [ \mathrm{tr}(AD^2u)ADu- AD^2uADu ]\cdot\nu\dd\mathcal{H}^{n-1}(x) \le 0.$$

 When $ADu\cdot\nu=0$ on $\partial\Omega$, in view of Lemma \ref{l3.yy1} we know  $$-[{\rm div}(AD u)ADu -  (AD u\cdot D)ADu ]\cdot\nu \le 0.$$
Hence, for the ${\bf K}(\phi)$ given by \eqref{K2}, we have
$$ {\bf K}(\phi)=- C
\int_{\partial\Omega}\phi^2(|\sqrt{A}Du|^2+\epsilon)^{p-2} [ {\rm div}(AD u)ADu -   (AD u\cdot D)ADu  ]\cdot\nu\dd\mathcal{H}^{n-1}(x) \le 0.$$
\end{proof}

\begin{proof}[Proof of Lemma \ref{boundaryterm} (ii).]
When $u=0$ on $\partial\Omega$, taking advantage of Lemma \ref{lem:3.3.1a},  we have
$$-[ \mathrm{tr}(AD^2u)ADu- AD^2uADu]\cdot\nu\le |\mathcal{B} | |ADu|^2.$$
Hence, for the ${\bf K}(\phi)$ given by \eqref{K1}, we have
\begin{align*} {\bf K}(\phi)&=- C
\int_{\partial\Omega}\phi^2(|\sqrt{A}Du|^2+\epsilon)^{p-2}[ \mathrm{tr}(AD^2u)ADu- AD^2uADu ]\cdot\nu\dd\mathcal{H}^{n-1}(x)  \\
&\le C \int_{\partial\Omega}\phi^2 |V_{A,\ez}(Du)|^2|\mathcal{B} | \dd\mathcal{H}^{n-1}(x).
\end{align*}

When $ADu\cdot\nu=0$ on $\partial\Omega$,  applying Lemma \ref{l3.yy1}, one has
$$-[{\rm div}(AD u)ADu-   (AD u\cdot D)ADu ]\cdot\nu\le |\mathcal{B} | |ADu|^2.$$
Hence, for the ${\bf K}(\phi)$ given by \eqref{K2}, we have
\begin{align*} {\bf K}(\phi)&=- C
\int_{\partial\Omega}\phi^2(|\sqrt{A}Du|^2+\epsilon)^{p-2}[{\rm div}(AD u)ADu-   (AD u\cdot D)ADu ]\cdot\nu\dd\mathcal{H}^{n-1}(x)  \\
&\le C \int_{\partial\Omega}\phi^2 |V_{A,\ez}(Du)|^2|\mathcal{B} | \dd\mathcal{H}^{n-1}(x).
\end{align*}

Finally, in light of Lemma \ref{CMineq} and letting $v=\phi V_{\epsilon}(ADu)$, in both cases  we have
 \begin{align*}
  |{\bf K}(\phi) | &\le C \int_{\partial\Omega}\phi^2 |V_{A,\ez}(Du)|^2|\mathcal{B} | \dd\mathcal{H}^{n-1}(x)\\
&\le C \Psi_{\mathcal B}(r)\|\phi   D V_{A,\ez}(Du) \|_{L^2(\Omega)}^2+ C \Psi_{\mathcal B}(r)\||D \phi   | V_{A,\ez}(Du)\|_{L^2(\Omega)}^2.
\end{align*}
 \end{proof}

\subsection {Proofs of Lemma \ref{l3.yy1} and Lemma \ref{lem:3.3.1a}}


To see Lemma \ref{l3.yy1} and Lemma \ref{lem:3.3.1a}, we need the following Lemma; and for its proof we refer to  for example \cite[3,1,1,6]{Grisvard}.  Here we omit the details.

\begin{lem} Assume that $u\in C^2(\overline\Omega)$. Then
\begin{align}\label{I1-3}
  &[\mathrm{div}(ADu)ADu -  (ADu\cdot D )ADu ] \cdot\nu\nonumber\\
  &=\mathrm{div}_{T}(ADu\cdot\nu(ADu)_{T})-(\mathrm{tr}\mathcal{B})(ADu\cdot\nu)^2\nonumber\\
  &\quad -\mathcal{B}((ADu)_{T},(ADu)_{T})-2(ADu)_{T}\cdot\nabla_{T}(ADu\cdot \nu) \quad\quad\textrm{on}\;\partial\Omega.
\end{align}
\end{lem}
Above   and below,   we  write $$
\mbox{$ (ADu)_{\nu}=\nu\cdot (ADu)$ and  $(ADu)_{T}=ADu-(ADu)_{\nu}\nu$.}$$
Using identity \eqref{I1-3} we prove Lemma \ref{lem:3.3.1a} and Lemma \ref{l3.yy1} as below.

\begin{proof}[Proof of Lemma \ref{l3.yy1}]
Since $ADu\cdot \nu=0$ on $\partial\Omega$, we  have
$\nabla_{T}(ADu\cdot \nu)=0$ and hence,
 $$\mathrm{div}_{T}(ADu\cdot\nu(ADu)_{T})=0\quad\mbox{on $\partial\Omega$}.$$
 We therefore obtain
$$
  - [\mathrm{div}(ADu)ADu - ((ADu\cdot D )ADu)] \cdot\nu
   =   \mathcal{B}((ADu)_{T},(ADu)_{T}) \quad\mbox{on $\partial\Omega$.}
$$
Thus
$$
 -[ \mathrm{div}(ADu)ADu - ((ADu\cdot D )ADu)] \cdot\nu
  \le |\mathcal{B}| |ADu|^2 \quad\mbox{on $\partial\Omega$.}
$$
If $\Omega$ is convex, we have $-\mathcal{B}\ge 0$, and hence
$$-[\mathrm{div}(ADu)ADu - ((ADu\cdot D )ADu) ]\cdot\nu \leq 0\quad\mbox{on $\partial\Omega$.}$$
\end{proof}

To prove Lemma \ref{lem:3.3.1a} we need the following auxiliary lemma.
\begin{lem}
Under the Dirichlet boundary condition $u=0$ on $\partial \Omega$, we have
\begin{align}\label{I1-3b}
-[\mathrm{tr}(AD^2u)ADu- AD^2uADu]\cdot\nu  =  \Big[A_{\nu,\nu}\mathrm{tr}(\mathcal{B} A_{T,T} )
+\mathcal{B}(A_{T,\nu}  ,A_{T,\nu}  ) \Big]\Big(\frac{\partial u}{\partial\nu}\Big)^2, \quad \mbox{on $\partial\Omega$},
\end{align}
where and below
 we always  set
$$ A_{\nu,\nu}:=\langle A\nu, \nu\rangle,\;\;  A_{T,\nu}:=A\nu-(A_{\nu,\nu})\nu,\;\; A_{T,T}= (\langle A\tau_{i},\tau_j\rangle)_{1\le i,j\le n-1}.$$
\end{lem}
\begin{proof}
 Recall that \eqref{diff} gives
\begin{align}\label{equa-ADu-1}
&[\mathrm{tr}(AD^2u)ADu- AD^2uADu]\cdot\nu \nonumber\\
&= [\mathrm{div}\,(ADu)ADu-(ADu\cdot D )ADu]\cdot\nu - \{[({\rm div} A) \cdot Du] ADu - [(ADu\cdot D)A]Du\}\cdot\nu.
\end{align}
It then suffices to prove that
\begin{align}\label{I1-3a}
  &[\mathrm{div}(ADu)ADu-  (ADu\cdot D )ADu] \cdot\nu\nonumber\\
  &= -\left[A_{\nu,\nu}\mathrm{tr}(\mathcal{B} A_{T,T} )
+\mathcal{B}(A_{T,\nu}  ,A_{T,\nu} ) \right]\Big(\frac{\partial u}{\partial\nu}\Big)^2
\nonumber\\
  &\quad +
 \sum^{n-1}_{k=1}\Big\{A_{\nu,\nu}\Big(\frac{\partial A}{\partial\tau_k}\nu\cdot\tau_k\Big)-\Big(\frac{\partial A}{\partial\tau_k}\nu\cdot\nu\Big)\tau_k\cdot A_{T,\nu}\Big\} \Big(\frac{\partial u}{\partial\nu}\Big)^2,
\end{align}
and
\begin{align}\label{I-3c}
&\{[(\mathrm{div} A)\cdot Du] ADu - [(ADu\cdot D)A]Du\}\cdot\nu\nonumber\\
&
  =\sum^{n-1}_{k=1}\Big\{A_{\nu,\nu}\Big(\frac{\partial A}{\partial\tau_k}\nu\cdot\tau_k\Big) -
  (A_{T,\nu}\cdot \tau_k)\Big(\frac{\partial A}{\partial\tau_k}\nu\cdot\nu\Big)\Big\} \Big(\frac{\partial u}{\partial\nu}\Big)^2 .
\end{align}

We show   \eqref{I1-3a} by  considering  all terms in its right side  in order.
Note  that $u=0$  on $\partial \Omega$, which implies that
 $\nabla_{T} u=0$, we have  $Du|_{\partial\Omega}=  \frac{\partial u}{\partial\nu}  \,\nu$,
\begin{align}\label{eq:ADu-nu}
(ADu)_\nu=\langle ADu,\nu\rangle =\langle A\nu,Du\rangle=A_{\nu,T}\cdot \nabla_{T}u+A_{\nu,\nu}\frac{\partial u}{\partial \nu}=A_{\nu,\nu} \frac{\partial u}{\partial\nu}
\end{align}
and
$$(ADu)_{T}=ADu-(ADu)_{\nu}\nu=A_{T,T}\cdot\nabla_{T}u+A_{T,\nu}\frac{\partial u}{\partial \nu} = A_{T,\nu}\frac{\partial u}{\partial \nu} . $$
Thus,
\begin{align*}
\mathrm{div}_{T}(ADu\cdot\nu(ADu)_{T}) &=\mathrm{div}_T
  \Big(A_{\nu,\nu} \Big(\frac{\partial u}{\partial\nu}\Big)  ^2A_{T,\nu} \Big)\\
&=A_{\nu,\nu} \Big(\frac{\partial u}{\partial\nu}\Big)  ^2  \mathrm{div}_T A_{T,\nu} +
  A_{T,\nu} \cdot   \nabla_T \Big( A_{\nu,\nu} \Big(\frac{\partial u}{\partial\nu}\Big)  ^2\Big),
\end{align*}
\begin{align*}
-(\mathrm{tr}\mathcal{B})(ADu\cdot\nu)^2
&= -A_{\nu,\nu} ^2 \Big(\frac{\partial u}{\partial\nu}\Big)  ^2\mathrm{tr} \mathcal{B} , \quad
  -\mathcal{B}((ADu)_{T},(ADu)_{T}) =  - \Big(\frac{\partial u}{\partial\nu}\Big)^2 \mathcal{B}  (A_{T,\nu}  ,A_{T,\nu}   ) ,
\end{align*}
and moreover,
\begin{align*}
-2(ADu)_{T}\cdot\nabla_{T}(ADu\cdot \nu)&=
   -2
\Big( \frac{\partial u}{\partial\nu}  A_{T,\nu}\Big)\cdot\nabla_{T}
  \Big(A_{\nu,\nu} \frac{\partial u}{\partial\nu} \Big) \\
&=-2
  A_{T,\nu}\cdot\nabla_{T}
  \Big(A_{\nu,\nu} \Big(\frac{\partial u}{\partial\nu}\Big)^2 \Big) +2  A_{T,\nu}\cdot
A_{\nu,\nu} \Big(\frac{\partial u}{\partial\nu}\Big) \nabla_{T}
  \Big(  \frac{\partial u}{\partial\nu} \Big) \\
&=-2
 A_{T,\nu}\cdot\nabla_{T}
  \Big(A_{\nu,\nu} \Big(\frac{\partial u}{\partial\nu}\Big)^2 \Big) +   A_{T,\nu}\cdot
A_{\nu,\nu}  \nabla_{T}
  \Big(  \Big(\frac{\partial u}{\partial\nu}\Big) ^2 \Big) \\
&= -  A_{T,\nu}\cdot\nabla_{T}
  \Big(A_{\nu,\nu} \Big(\frac{\partial u}{\partial\nu}\Big)^2 \Big)-    \Big(\frac{\partial u}{\partial\nu}\Big) ^2  A_{T,\nu} \cdot
\nabla_{T}A_{\nu,\nu} .
 \end{align*}
Combining them together, we obtain
\begin{align} \label{e3.xx1}
  &\mathrm{div}(ADu)ADu\cdot\nu- ((ADu\cdot D )ADu) \cdot\nu\nonumber\\
  &=    \Big(\frac{\partial u}{\partial\nu}\Big)^2[ -A_{\nu,\nu} ^2  \mathrm{tr}(\mathcal{B}) -\mathcal{B}  (A_{T,\nu}  ,A_{T,\nu}   )  + A_{\nu,\nu}\mathrm{div}_TA_{T,\nu}  -  A_{T,\nu}\cdot \nabla_{T}A_{\nu,\nu}].
\end{align}

 Note  that  by $A_{T,\nu}=A\nu-A_{\nu,\nu}\cdot\nu$, we  have
\begin{align}\label{div}
 A_{\nu,\nu}  \mathrm{div}_T(A_{T,\nu})&=A_{\nu,\nu}\sum^{n-1}_{k=1}\Big(\frac{\partial}{\partial\tau_{k}}A_{T,\nu}\Big)\cdot\tau_k\nonumber\\
 &=A_{\nu,\nu}\sum^{n-1}_{k=1}\Big(\frac{\partial}{\partial\tau_{k}}\left\{A\nu-A_{\nu,\nu}\cdot\nu\right\}\Big)\cdot\tau_k\nonumber\\
  &=A_{\nu,\nu}\Big[\sum^{n-1}_{k=1}\Big(A\frac{\partial\nu}{\partial\tau_k}\cdot\tau_k+\frac{\partial A}{\partial \tau_k}\nu\cdot\tau_k-A_{\nu,\nu}\frac{\partial\nu}{\partial\tau_k}\cdot\tau_k\Big)\Big]\nonumber\\
  &=A_{\nu,\nu}\Big[A_{\nu,\nu}\mathrm{tr}\mathcal{B}+\sum^{n-1}_{i,k=1}\frac{\partial\nu}{\partial\tau_k}(A_{T,T})_{i,k}\cdot\tau_i+\sum^{n-1}_{k=1}\frac{\partial A}{\partial\tau_k}\nu\cdot\tau_k\Big]\nonumber\\
  &=A_{\nu,\nu}^2\mathrm{tr}\mathcal{B}-A_{\nu,\nu}\mathrm{tr}(\mathcal{B}A_{T,T})+A_{\nu,\nu}\sum^{n-1}_{k=1}\frac{\partial A}{\partial\tau_k}\nu\cdot\tau_k.
\end{align}
On the other hand,
By $A_{\nu,\nu}=\nu\cdot A\nu$, we also have
\begin{align}
  \frac{\partial}{\partial\tau_k}A_{\nu,\nu}&=2A\frac{\partial\nu}{\partial\tau_k}\cdot\nu+\frac{\partial A}{\partial\tau_k}\nu\cdot\nu=2 \frac{\partial\nu}{\partial\tau_k}\cdot A\nu+\frac{\partial A}{\partial\tau_k}\nu\cdot\nu.\nonumber \end{align}
Since $ \frac{\partial\nu}{\partial\tau_k}\cdot\nu=0$ and  $A_{T,\nu}=A\nu-A_{\nu,\nu}\cdot\nu$ we have

\begin{align}
\frac{\partial}{\partial\tau_k}A_{\nu,\nu}&=2 \frac{\partial\nu}{\partial\tau_k}\cdot  A_{ T, \nu} + \frac{\partial A}{\partial\tau_k}\nu\cdot\nu, 
\end{align}
which directly yields
\begin{align}\label{nubla}
 -A_{T,\nu}\cdot \nabla_{T}  A_{\nu,\nu}=-\sum_{k=1}^{n-1} (A_{T,\nu}\cdot\tau_k)  \frac{\partial   A_{\nu,\nu}}{\partial\tau_k}=  2\mathcal{B}( A_{T,\nu },A_{T,\nu})-\sum^{n-1}_{k=1}\left(\frac{\partial A}{\partial\tau_k}\nu\cdot\nu\right)   (A_{T,\nu}\cdot   \tau_k ).
\end{align}

Plugging the above identity and  \eqref{div}  into \eqref{e3.xx1}, we get the desired identity \eqref{I1-3a}.

Finally, by $D_Tu=0$ on $ \partial\Omega$ we have
  $$(\mathrm{div} A)\cdot Du= (\mathrm{div} A)\cdot \nu \frac{\partial u}{\partial \nu}=   \Big[\Big\langle\frac{\partial  A}{\partial \nu} \nu,\nu\Big\rangle  +
 \sum_{k=1}^{n-1} \Big\langle\frac{\partial A}{\partial \tau_k}  \nu,\tau_k\Big\rangle \Big]\frac{\partial u}{\partial \nu}  \quad\mbox{on $\partial\Omega$}.$$
By $ADu \cdot \nu= A_{\nu,\nu} \frac{\partial u}{\partial \nu}$ on $\partial\Omega$, we obtain
\begin{align} \label{xyyy-1}
    [(\mathrm{div} A)\cdot Du] ADu \cdot \nu
   & = \Big[\Big\langle\frac{\partial  A}{\partial \nu} \nu,\nu\Big\rangle  +
 \sum_{k=1}^{n-1} \Big\langle\frac{\partial A}{\partial \tau_k}  \nu,\tau_k\Big\rangle \Big]
   A_{\nu,\nu} \Big(\frac{\partial u}{\partial \nu}\Big)^2 \quad\mbox{on $\partial\Omega$}.
\end{align}
Moreover, by $D_Tu=0$ on $ \partial\Omega$, we have $$(ADu\cdot D)A =  A_{\nu,\nu}\frac{\partial u}{\partial \nu}
  \frac{\partial A }{\partial \nu}
  + \sum_{k=1}^{n-1} \frac{\partial u}{\partial \nu} (A_{T,\nu}\cdot \tau_k)\frac{\partial A}{\partial \tau_k}  \quad\mbox{on $\partial\Omega$},$$

 $$[(ADu\cdot D)A]Du = A_{\nu,\nu}\Big(\frac{\partial u}{\partial \nu}\Big)^2 \frac{\partial A}{\partial\nu}\nu + \sum_{k=1}^{n-1} \Big(\frac{\partial u}{\partial \nu}\Big)^2 (A_{T,\nu}\cdot \tau_k)\frac{\partial A}{\partial \tau_k} \nu,$$
 and
\begin{align*}
[(ADu\cdot D)A]Du \cdot \nu
  & = A_{\nu,\nu}\Big(\frac{\partial u}{\partial \nu}\Big)^2    \Big\langle\frac{\partial  A}{\partial \nu} \nu,\nu\Big\rangle  +
\sum_{k=1}^{n-1} \Big(\frac{\partial u}{\partial \nu}\Big)^2 (A_{T,\nu}\cdot \tau_k)\Big\langle \frac{\partial A}{\partial \tau_k} \nu,\nu\Big\rangle .
\end{align*}
From the above identity and \eqref{xyyy-1} we conclude \eqref{I-3c}.

 \end{proof}

Now we are ready to prove Lemma \ref{lem:3.3.1a}.

\begin{proof}[Proof of Lemma \ref{lem:3.3.1a}]
Recall that
\begin{align}\label{I1-3b}
-[\mathrm{tr}(AD^2u)ADu- AD^2uADu]\cdot\nu  =  \left[A_{\nu,\nu}\mathrm{tr}(\mathcal{B} A_{T,T} )
+\mathcal{B}(A_{T,\nu}  ,A_{T,\nu}  ) \right]\Big(\frac{\partial u}{\partial\nu}\Big)^2,
\end{align}
We obviously have
$$-
[\mathrm{tr}(AD^2u)ADu- AD^2uADu]\cdot\nu \le C|\mathcal B||ADu|^2$$
for some constant $C$ depending on $L$, that is, \eqref{lem5.1-2} holds.

Moreover,
if $\Omega$ is convex,   we know  that
  $  -\mathcal B\ge0$, and hence  $\mathcal{B} (A_{T,\nu}  ,A_{T,\nu}  )\le0$.
Since  $$  B_{n-1} :=  \left(-\mathcal B(\tau_i,\tau_j) \right)_{1\le i,j\le n-1}
\ge 0,$$
writing $B:={\rm diag}\,\{B_{n-1},0\} $, we also have $B\ge 0$.
 Since $A\ge 0$,  we have  $A_{\nu,\nu}\ge 0$ and $\sqrt AB\sqrt A\ge0$, and hence $\mathrm{tr}(\sqrt AB\sqrt A)\ge0$.
Note that $$\mathrm{tr}(B_{n-1}A_{T,T})=\mathrm{tr}(BA)=\mathrm{tr}(\sqrt AB\sqrt A) $$
and that
$A_{\nu,\nu}\mathrm{tr}(\mathcal{B} A_{T,T} )=A_{\nu,\nu}\mathrm{tr}(B_{n-1}A_{T,T})$,
we conclude that $A_{\nu,\nu}\mathrm{tr}(\mathcal{B} A_{T,T} )\le0$.
Thus by \eqref{I1-3b} we have
  $$-
[\mathrm{tr}(AD^2u)ADu- AD^2uADu]\cdot\nu \le 0,$$
that is, \eqref{lem5.1-1} holds.


\end{proof}

\section{Proofs of key Lemmas \ref{lem:Veps-es1} \& \ref{cmineq2}} 

Recall   the following Garliardo-Nirenberg-Sobolev  inequality

\begin{lem}\label{G-N-S} For any $1\le s<\frac{n+2}{n-2}$ and $\theta=\frac{2n}{n+2}\frac{s-1}s$, or $s=\frac{n+2}{n-2}$ with $n\ge3$ and $\theta=1$, we have
$$\|v-v_\Omega\|^2_{L^{s}(\Omega)}\le C\|Dv\|_{L^{2}(\Omega)}^{2\theta}\|v\|_{L^1(\Omega)}^{2(1-\theta)}\quad\forall v\in W^{1,2}(\Omega),$$
where $v_\Omega=\frac{1}{|\Omega|}\int_\Omega v\dd x$, and $C>0$ is a constant depends on $s,n$
and  $\lip_\Omega$.
\end{lem}


%
%
%
%
%
%

\begin{proof}[proof of Lemma \ref{lem:Veps-es1}]
(i) Let $s=\frac{2q}{q-2}$, that is, $q=\frac{2s}{s-2}$. Note that $q>n\ge2$ implies that $2<s=\frac{2q}{q-2}<\frac{2n}{n-2}$.
By the H\"older inequality, one has
 \begin{align*}
  \int_{\Omega}  |DA|^2|V_{A,\ez}(Du)|^2  \dd x
   & \le
  \Big( \int_{\Omega\cap B(x,r)} |DA|^{\frac {2s}{s-2} }\Big)^{ \frac {s-2}s} \Big(\int_{\Omega}|  V_{A,\ez}(Du)|^s  \dd x\Big)^{2/s}\\
&=  \|DA\|^2_{ L^{q }(\Omega \cap B(x,r))} \|  V_{A,\ez}(Du)\|_{L^s(\Omega)}^{2}.
  \end{align*}
Write  $ \theta=\frac{2n}{n+2}\frac{s-1}s $ and note $0<\theta<1$.
By the Garliardo-Nirenberg-Sobolev   inequality  we have
$$\|V_{A,\ez}( A Du)-[V_{A,\ez}( A Du)]_\Omega\|^2_{L^{s}(\Omega)}\le C\|DV_{A,\ez}( A Du)\|_{L^{2}(\Omega)}^{2\theta}\|V_{A,\ez}( A Du)\|_{L^1(\Omega)}^{2(1-\theta)}.$$
Applying  Young's inequality, for any $\eta\in(0,1)$ we have
$$\|V_{A,\ez}( A Du)-[V_{A,\ez}( A Du)]_\Omega\|^2_{L^{s}(\Omega)}\le  \eta \|DV_{A,\ez}( A Du)\|^2_{L^{2}(\Omega)}+ \frac C\eta\|V_{A,\ez}( A Du)\|^2_{L^1(\Omega)}.$$
Thus
  \begin{align*}
    \| V_{A,\ez}(Du)\|_{L^s(\Omega)}^{2}
&\le \eta\|D   V_{A,\ez}(Du) \|^2_{L^2(\Omega)}+ \frac C\eta \|  V_{A,\ez}(Du) \|_{L^1(\Omega)}^2.
  \end{align*}

(ii) For any $\phi\in C^\fz_c(B(z,r))$, by the H\"older inequality, we have
 \begin{align*}
  \int_{\Omega} |DA|^2|V_{A,\ez}(Du)|^2 \phi^2\dd x
   & \le   \|DA\|^2_{ L^{n }(\Omega \cap B(z,r))} \|\phi  V_{A,\ez}(Du)\|_{L^{\frac{2n}{n-2}}(\Omega)}^{2}.
  \end{align*}
Recall the Sobolev imbedding gives
 \begin{align*}
\|\phi  V_{A,\ez}(Du)-[\phi  V_{A,\ez}(Du)]_\Omega\| _{L^{\frac{2n}{n-2}}(\Omega)}\le C\|D[\phi  V_{A,\ez}(Du)]\|_{L^2(\Omega)},\end{align*}
we have
 \begin{align*}
\|\phi  V_{A,\ez}(Du) \|^2 _{L^{\frac{2n}{n-2}}(\Omega)}&\le C\|D[\phi  V_{A,\ez}(Du)]\|^2_{L^2(\Omega)}+ \|\phi  V_{A,\ez}(Du)\|^2_{L^1(\Omega)}\\
&\le C\|\phi D  V_{A,\ez}(Du) \|^2_{L^2(\Omega)}+  C\| |D\phi| | V_{A,\ez}(Du)| \|^2_{L^2(\Omega)} +  \|\phi  V_{A,\ez}(Du)\|^2_{L^1(\Omega)} \end{align*}.
\end{proof}

Below we prove Lemma \ref{cmineq2}.

\begin{proof}[Proof of Lemma \ref{cmineq2}]
Applying Lemma \ref{G-N-S} for $v=V_{A,\epsilon}(Du)$ with $s=2$ and $ \theta=\frac{n}{n+1}$, in light of Young's inequality we obtain
\begin{align*}
\|  V_{A,\ez}(Du)-[   V_{A,\ez}(Du)]_\Omega\| _{L^2 (\Omega)}\le \eta\|D V_{A,\ez}(Du)\|_{L^2(\Omega)}+\frac  C\eta \|V_{A,\ez}(Du)\|_{L^1(\Omega)}\end{align*}
for any $\eta\in(0,1)$. This then gives the desired inequality.
\end{proof}

\section{Proofs  of Theorems \ref{thm:convex} and \ref{thm}}


To prove Theorems \ref{thm:convex} and \ref{thm} we need the following auxiliary lemmas.

The first is the following extension lemma,
whose proof is given in appendix, where we adapt some arguments of Sobolev extension operator by Jones and also \cite{KZZ}.
\begin{lem}\label{Aconve} Suppose that $\Omega$ is a bounded Lipschitz domain.
Let  $A\in\mathcal E_L(\Omega)$ satisfy \eqref{sobA}.
 There exists a family    $\{A^\ez\}_{\ez\in(0,1]}\subset\mathcal E_L(\rn)\cap C^\fz(\rn)$ such that

(i) if $DA\in L^q(\Omega)$ with $q> n\ge 2$,  we have
 $A^\ez\to A \in W^{1,q}(\Omega)$ and
$\|D   A^\ez\|_{L^q(\rn)}\le C \|DA\|_{L^q(\Omega)} $ for some constant $C$ depending $n,q$ and $\Omega$.

(ii) if  $DA\in L^n(\Omega)$ with $ n\ge 2$, we have  $A^\ez\to A \in W^{1,n}(\Omega)$,  $\|D  A^\ez \|_{L^n(\rn)}<C\|DA\|_{L^n(\Omega)} $ and
 $$\Phi_{  A^\ez,\Omega^t}(r):=\sup_{x\in\overline \Omega_t}\|DA^\ez\|_{L^n(B(x,r)\cap\Omega^t)}\le C \Phi_{  A}( C r),\quad \forall 0<t,\ez< r<\diam\Omega $$
for some constant $C$ depending on  $n$ and $\Omega$.   Here and below $$\Omega^t:=\{x\in\rn,\dist(x,\Omega)<t\}.$$
\end{lem}

Next we recall the following two approximation results of domains; see for example \cite{CMaz}.
\begin{lem}\label{appconvex}
Given any bounded convex domain $\Omega$ in $\rn$,
there is a  sequence $\{\Omega_k\}$ of smooth  bounded convex domains in $\rn$ such that
\begin{equation}\label{e3.xx2}\mbox{
$\Omega\Subset\Omega_{k}$, $\lim_{k\rightarrow\infty}|\Omega_{k}\setminus\Omega|=0$
and $\lim_{k\rightarrow\infty} d_H(\Omega_k,\Omega)=0$. }
\end{equation}
\end{lem}

\begin{lem}\label{appdomain}
Given any    Lipschitz domain $\Omega$  satisfying \eqref{exx1}, 
there is a  sequence $\{\Omega_k\}$ of smooth   domains in $\rn$ such that
\begin{equation}\label{e3.xx2}\mbox{
$\Omega\Subset\Omega_{k}$, $\lim_{k\rightarrow\infty}|\Omega_{k}\setminus\Omega|=0$
and $\lim_{k\rightarrow\infty} d_H(\Omega_k,\Omega)=0$, $\Phi_{\Omega_k}(r)\le C\Phi_{\Omega} (Cr)$,}
\end{equation}
where $C$ is a constant depending only in $n$ and $\Omega$.
\end{lem}

Moreover the following two approximation solutions to certain equations are needed.
Recall that notions  $V_{ A,\ez }(Du)$ and  $V_{  \ez }(Du)$  in Section 2,  we write $V_{A}(Du)=V_{A,0}(Du)$ and $V(Du)= V_{0}(Du)$.

\begin{lem}\label{appconv}
Let $\Omega_\fz=\Omega$ and $ \{\Omega_k\}_{k\in\nn}$ be as Lemma \ref{appconvex}, or be as in Lemma \ref{appdomain}.
Let $g\in C^\fz_c(\Omega)$, $A\in C^\fz(\rn)\cap \mathcal E_L(\rn)$ and $\ez\in(0,1]$.
  For $k\in\nn\cup\{\fz\}$, denote by $v_k\in W^{1,p} (\Omega_k)$ the weak solution to $\mathcal L_{A,\ez,p}v_k=g$ in $\Omega_k$ with Dirichlet $0$-boundary
or Neumann $0$-boundary.
Then  
$V_{A,\ez}(Dv_k)\to V_{A,\ez}(D v_\fz)$ and $V_{ \ez}(Dv_k)\to V_{ \ez}(D v_\fz)$  in $\Omega_\fz$ as $\nn\ni k\to\fz$.

\end{lem}

\begin{lem}\label{approx} Let $\Omega \subset\rn$ be a bounded Lipschitz domain, $g\in C^\fz_c(\Omega )$ and
$\{A^\ez\}_{\ez\in[0,1]}\in\mathcal E_L(\Omega )$  with $A^\ez\to A^0$ in as $\ez\to0$.
 For $\ez\in[0,1]$, denote by $v_\ez\in W^{1,p}(\Omega)$ be a weak solution to $\mathcal L_{A^\ez,\ez,p} v_\ez =g $ in $\Omega $
with Dirichlet $0$-boundary or 
with  Neumann $0$-boundary. 
Then
$V_{A,\ez}(Dv_\ez )\to V_{A^0 }(D v_0)$ and $V_{ \ez}(Dv_\ez)\to V(D v_0)$ almost everywhere in $\Omega$ as $0<\ez \to0$.

\end{lem}


\begin{proof}[Proof of Theorem \ref{thm:convex}]
 We only consider the case Dirichlet $0$-boundary; the case  Neumann $0$-boundary
follows  exactly the same argument.
Let $\Omega$ be a bounded convex domain, $A\in \mathcal E_L(\Omega)$ satisfying \eqref{sobA},  and $f\in L^2(\Omega)$. Let $u\in W^{1,2}_0(\Omega)$ be the unique generalized solution
to $\mathcal L_{A,p}u=f$ in $\Omega$ with Dirichlet  $0$-boundary.
We prove \eqref{2nd-reg} as below. Note that \eqref{2nd-reg} reads as
\begin{equation}\label{2nd-reg-1}
\mbox{$V_A(Du), V(Du)\in W^{1,2}(\Omega)$
with   $ \|D V_A(Du)\|_{L^2(\Omega)}+ \|DV(Du)\|_{L^2(\Omega)}\le C\|f\|_{L^{2}(\Omega)}$}
\end{equation}

We choose $\{f^\ell\}_{\ell\in\nn}\subset C^\fz_c(\Omega)$ so that
 $\|f^\ell\|_{L^2(\Omega)}\le 2 \|f\|_{L^2(\Omega)}$ for all   $\ell$ and
 $\|f^\ell-f\|_{L^2(\Omega)}\to 0$ as $\ell\to\fz$.
 Let  $\{A^\ez\}_{\ez\in(0,1]} $ be as in Lemma \ref{Aconve}
 and  $\{\Omega_k\}_{k\in\nn}$ be as in Lemma \ref{appconv}.
Given any $\ell\in\nn,$ $\ez\in(0,1)$ and $k\in\nn$, let  $u_{(\ell,\ez,k)}$ be the smooth solution to the problem
\begin{equation}\label{AppDiri}
\mathcal L_{A^\ez,\ez,p} u_{(\ell,\ez,k)}=f^\ell\quad \textrm{in}\;\; \Omega_k;\\
     u_{(\ell,\ez,k)} =0\quad\textrm{on}\;\;\partial\Omega_k.
\end{equation}
Observe that by Lemma \ref{Aconve} and Lemma \ref{appconv}, we have
\begin{enumerate}
\item[(i)]  If $DA\in L^q(\Omega)$ for some $q>n\ge 2$,  writing $R_\sharp=C_{\rm ext,q}(\Omega)\|A\|_{L^q(\Omega)}$,  we have
$\|DA^\ez\|_{L^q(\Omega_k)}\le  R_\sharp$ for all $ k$.

\item[(ii)]  If $DA\in L^n(\Omega)$ with $n\ge 3$, let $\dz_\sharp$ be as in  Theorem \ref{thm:convex-smooth} (ii).   Then
$\Phi_{A,\Omega}(r_A)< \dz_\sharp/C_{\rm ext,n}(\Omega)$ for some $r_A>0$.
Let $r_\sharp=r_A/C_{\rm ext,n}(\Omega)$.
Then
$\Phi_{A^\ez,\Omega_k}(r_\sharp)< \dz_\sharp $ whenever $\ez<r_\sharp$ and $k$ large such that $d_H(\Omega_k,\Omega)\le r_\sharp$.
\end{enumerate}
Thanks to this, for all $\ell\in\nn$, for all sufficiently small $\ez>0$ and all sufficiently large $k$,
 we apply  Theorem \ref{thm:convex-smooth} to $u_{(\ell,\ez,k)}$ so to obtain
\begin{align}\label{estimate1}
\|D V_{A^\ez,\ez}(Du_{(\ell,\ez,k)})\|_{L^2(\Omega_k)}+ \|DV_{\ez}(Du_{(\ell,\ez,k)})\|_{L^2(\Omega_k)}&\leq C_{\star} \|f^\ell\|_{L^2(\Omega_k)} \le 2 C_{\star}\|f\|_{L^2(\Omega)}.
\end{align}
where the constant $ C_{\star}$ is as determined by Theorem \ref{thm:convex-smooth}, in particular, independent of $\ez,\ell,k$. 
From this, we conclude the desired result \eqref{2nd-reg-1}  by sending $k\to\fz$, $\ez\to0$,  and $\ell\to \fz$ in order.
The details are given as below.

\noindent {\bf Send $k\to\fz$.}  Fix any $\ell\in\nn$ and any sufficiently small $\ez>0$, from \eqref {estimate1} one deduces that
 $V_{A^\ez,\ez}(Du_{(\ell, \ez,k)})\in W^{1,2}(\Omega)$ and
$V_{\ez}(Du_{(\ell, \ez,k)})\in W^{1,2}(\Omega)$, both of which are uniform in all sufficiently large  $k$.
By  the compactness of Sobolev space, we know that
$V_{A^\ez,\ez}(Du_{(\ell,\ez,k)})$
 converges to some function $G \in W^{1,2}(\Omega)$
and
$V_{\ez}(Du_{(\ell,\ez,k)})$ converges to $\wz G$
in $L^2(\Omega)$ and weakly in $W^{1,2}(\Omega)$, and

$$\|D  G  \|_{L^2(\Omega)}+ \|D \wz G  \|_{L^2(\Omega)}\le
\liminf_{k\to\fz}[\|D V_{A^\ez,\ez}(Du_{(\ell,\ez,k)})\|_{L^2(\Omega)}+ \|D V_{\ez}(Du_{(\ell,\ez,k)})\|_{L^2(\Omega)}]\le  2C_{\star}\|f\|_{L^2(\Omega)}.$$

On the other hand, denote by
$u_{(\ell,\ez)}\in W^{1,p}_0(\Omega)$  the  weak solution to the
 equation   $ \mathcal L_{A^\ez,\ez,p} u_{(\ell,\ez)}
 =f^\ell $ in $\Omega$
  Dirichlet 0-boundary.   By Lemma 7.4 one has $V_{A^\ez,\ez}(Du_{(\ell, \ez,k)})\to  V_{A^\ez,\ez}(Du_{(\ell, \ez )})$ almost verywhere in $\Omega$ as $k\to\fz$, and also
 $V_{ \ez}(Du_{(\ell, \ez,k)})\to  V_{ \ez}(Du_{(\ell, \ez )})$ almost everywhere in $\Omega$ as $k\to\fz$.
Thus $G=
 V_{A^\ez,\ez}(Du_{(\ell, \ez )})\in W^{1,2}(\Omega)$
and  $\wz G=
 V_{A^\ez,\ez}(Du_{(\ell, \ez )})\in W^{1,2}(\Omega)$, and hence
 and
\begin{equation}\label{uellez}\|D   V_{A^\ez,\ez}(Du_{(\ell, \ez )})\|_{L^2(\Omega)}+
\|D  V(Du_{(\ell, \ez )})\|_{L^2(\Omega)} \le  2C_{\star}\|f\|_{L^2(\Omega)}.\end{equation}

\noindent {\bf Send $\ez\to\fz$.}
Given any sufficiently large $  \ell$,   denote by   $u_{(\ell)}\in W^{1,p}(\Omega)$   the unique  weak solution to the
 equation  $
\mathcal L_{A,p} u_{( \ell)}
=f^\ell $ in $\Omega$ with Dirichlet 0-boundary.
By Lemma 7.5,
$V_{A^\ez,\ez}(Du_{(\ell,\ez)})\to  V_{A }(Du_{( \ell)})$ and
$V_{\ez}(Du_{(\ell,\ez)})\to  V(Du_{( \ell)})$ almost everywhere  as $\ez\to0$ (up to some subsequence)
Thanks to this, \eqref {uellez} and  the compactness of Sobolev space $W^{1,2}(\Omega)$, we know that
$V_{A^\ez,\ez}(Du_{(\ell,\ez)})$
 converges to some function $V_{A }(Du_{( \ell)})$
and
$V_{\ez}(Du_{(\ell,\ez)})$ converges to $V (Du_{( \ell)})$
in $L^2(\Omega)$ and weakly in $W^{1,2}(\Omega)$,
and
\begin{align}\label{uell}&\|D   V_{A }(Du_{(\ell  )})\|_{L^2(\Omega)}+
\|D  V_{ \ez}(Du_{(\ell  )})\|_{L^2(\Omega)}\nonumber\\
&\le
\liminf_{\ez\to0}[\|D V_{A^\ez,\ez}(Du_{(\ell,\ez)})\|_{L^2(\Omega)}+ \|D V(Du_{(\ell,\ez)})\|_{L^2(\Omega)}]\le  2C_{\star}\|f\|_{L^2(\Omega)}.\end{align}

\noindent {\bf Send $\ell\to\fz$.}
 Sending $\ell\to\fz$,  by \cite{CMaz2017}, we have
$u_{(  \ell)} \to   u $  in $W^{1,p}(\Omega)$. This yields that
$V_{A^\ez,\ez}(Du_{(\ell )})\to  V_{A }(Du )$ almost everywhere  as $\ez\to0$ (up to some subsequence)
and
$V_{\ez}(Du_{(\ell )})\to  V(Du )$ almost everywhere  as $\ell\to0$.
By this, \eqref{uell} and   the compactness of Sobolev space,
$V_{A^\ez,\ez}(Du_{(\ell )})$
 converges to some function $V_{A }(Du )$
and
$V_{\ez}(Du_{(\ell )})$ converges to $V (Du )$
in $L^2(\Omega)$ and weakly in $W^{1,2}(\Omega)$,
and
\begin{align*}\label{u}&\|D   V_{A }(Du )\|_{L^2(\Omega)}+
\|D  V(Du )\|_{L^2(\Omega)}\nonumber\\
&\le
\liminf_{\ell\to\fz}[\|D V_{A}(Du_{(\ell )})\|_{L^2(\Omega)}+ \|D V(Du_{(\ell )})\|_{L^2(\Omega)}]\le  2C_{\star}\|f\|_{L^2(\Omega)} \end{align*}
as desired.
 \end{proof}


\begin{proof}[Proof of Theorem \ref{thm}]
Given any $\Omega$ satisfying \eqref{exx1}, let $\Omega_k$ be as in Lemma \ref{appdomain}.  Let  $\dz_\ast$ be as in Theorem 1.6 and  $C_\ast$ be as in Lemma \ref{appdomain}.
If  $\Psi_{\Omega}(C_\ast r_\ast)\le \delta_\ast/C_\ast$, then  $\Psi_{\Omega_k}( r_\ast)\le \delta_\ast $.
Following the procedure for the proof of Theorem 1.1 we will get Theorem 1.2. We omit the details.
\end{proof}

\subsection{Proofs of Lemmas \ref{appconv} and \ref{approx}}

\begin{proof}[Proof of Lemma \ref{appconv}]
{\it Case Dirichlet $0$-boundary.}  It suffices to prove $v_k\to v_\fz$ in $C^{1,\alpha}(\Omega_\fz)$ as $\ez\to0$.
Firstly, we show that,  for $k\in\nn$,
\begin{equation}\label{vez1}
  \int_{\Omega_k }|\sqrt{A^\ez }Dv_k|^p\,\dd x \leq C(n,p,L,\Omega_1)
  \int_{\Omega }(|g|^{p'}+1)\,\dd x.
\end{equation}
Indeed, since $v_k\in W^{1,p}_0(\Omega_k)$ is a weak solution to
  $\mathcal L_{A^\ez,\ez,p} v_k=g $ in $\Omega_k$ we obtain
\begin{equation} \label{vez2}
\int_{\Omega_k } (\sqrt{A  })^{-1}V_{A ,\ez}(Dv_k)\cdot \sqrt {A} Dv_k\,\dd x =
\int_{\Omega_k }  V_{A ,\ez}(Dv_k)\cdot Dv_\ez\,\dd x =\int_{\Omega_k}g  v_k\,\dd x= \int_{\Omega_\fz}g  v_k\,\dd x.
\end{equation}
 Observe that
 $$|\sqrt{A^\ez }Dv_k|^p\le (\sqrt{A  })^{-1}V_{A ,\ez}(Dv_k)\cdot \sqrt {A} Dv_k
 +C(p)\ez^{\frac p2}\quad \forall \ez\in (0,1).$$
 Thanks to this, applying Young's inequality with $\eta\in(0,1)$   we have
\begin{equation}\label{vezz}
  \int_{\Omega_k }|\sqrt{A^\ez }Dv_k|^p\,\dd x \leq \frac4{\eta}\int_{\Omega }|g|^{p'}\,\dd x+\eta\int_{\Omega_k}|v_k|^p\,\dd x+C(p,\Omega_1).
\end{equation}
Since $v_k\in W^{1,p}_0(\Omega_1)$, where  $v_k$ is extended to the Lipchitz domain $\Omega_1$ by setting $v_k= 0$ on $\Omega_1\setminus \Omega_k$,
by Poincar\'e's inequality we have
\begin{equation} \label{vez3}\int_{\Omega_k}|v_k|^p\,\dd x=\int_{\Omega_1}|v_k|^p\,\dd x \le C(n,p,\Omega_1)\int_{\Omega_1}|D v_k|^p\,\dd x =  C(n,p,\Omega_1)\int_{\Omega_k}|D v_k|^p\,\dd x. \end{equation}
Since  $\frac1L\leq\sqrt{A^\ez }\leq L $,   choosing $\eta$ smooth enough so that
$$C(n,p,\Omega_1)\eta\int_{\Omega_k  }|Dv_k|^p\,\dd x\le \frac12 \int_{\Omega_k }|\sqrt{A^\ez }Dv_k|^p\,\dd x.$$
From this and \eqref{vezz} we conclude  \eqref{vez1}.

On the other hand, it is well-known that,  there exists $\alpha>0$ such that,
for  any smooth subdomain $ U\Subset\Omega_\fz$,
 $v_k\in C^{1,\alpha}(\overline U)$  uniformly in  all $k\in\nn$.   Moreover, observe that $\Omega_k$ satisfies the uniform regular condition uniformly in $k$, that is,
$$|B(x,r)\setminus \Omega|\ge c|B(x,r)|\quad\forall x\in \partial\Omega_k,\forall k\in\nn$$
for some constant $c>0$. There exists some $\bz>0$ such that $v_k\in C^{0,\bz}(\Omega_k)$ with the norm $\sup_{k\in\nn}\|v_k\|_{C^{0,\bz}(\overline{\Omega_k})}<\fz$; see    \cite{LU1,LadU68,Tr}.
Since $v_k|_{\partial \Omega_k}=0$ and $d_H(\Omega,\Omega_k)\to0$ we know that $v_k|_{\partial\Omega}\to0$ uniformly as $\nn\ni k\to\fz$.
Thus
We can find  a function $v\in C^{0,\bz}(\overline\Omega_\fz)\cap C^{1,\alpha}(\Omega_\fz)$ such that
  $v_k\to v $ in $C^{0,\bz}(\overline\Omega)$ and $Dv_k\to Dv$ in $ C^{0,\alpha}(\Omega_\fz)$      as $ \nn\ni k\to\infty$ (up to some subsequence).
Consequently, one has
$v_\fz|_{\partial\Omega}=0$ and $v \in C^{0,\bz}(\overline \Omega)$, and moreover,
$\frac1{r^n}\int_{B(x,r)\cap\Omega}|v |\,\dd x=0$ as $r\to0$ for all $x\in\partial\Omega$. By \cite{SZ} and $v\in W^{1,p}(\Omega_\fz)$ one conclude  $v \in W^{1,p}_0(\Omega)$.

 Next,  as $k\to\fz$, since $Dv_k\to Dv$ in $ C^{0,\alpha}(\Omega_\fz)$,  we have    $V_{A ,\ez}(Dv_k)\to V_{A,\ez}(Dv)
$  in $C^{0,\alpha}(\Omega_\fz)$.
Thus for any $\phi\in C^\fz_c(\Omega_\fz)$, it follows that
\begin{align*}\int_{\Omega_\fz} V_{A^\ez,\ez}(Dv)\cdot D\phi\,\dd x&=\lim_{k\to\fz}\int_{\Omega_\fz}
V_{A^\ez,\ez}(Dv_k)\cdot D\phi\,\dd x\\
&= \lim_{k\to\fz}\int_{\Omega_k}
V_{A^\ez,\ez}(Dv_k)\cdot D\phi\,\dd x
  = \lim_{k\to\fz}\int_{\Omega_k} g \phi\,\dd x
 =\int_{\Omega_\fz} g \phi\,\dd x.
\end{align*}
Observe that \eqref{vez1} implies that
  $V_{A,\ez} (Dv_k)\in L^{p'}(\Omega_\fz)$ uniformly in $k\in\nn$.
By a density argument one has
\begin{align*}\int_{\Omega_\fz} V_{A^\ez,\ez}(Dv)\cdot D\phi\,\dd x&
 =\int_{\Omega_\fz} g\phi\,\dd x \mbox{ for all $\phi\in W^{1,p'}_0(\Omega_\fz) $, }
\end{align*}
 that is,  $\mathcal L_{A^\ez,\ez,p} v=g$  in $ \Omega_\fz$ in weak sense.
By the uniqueness of solutions to
 the equation $\mathcal L_{A^\ez,\ez,p} v=g$  in $ \Omega_\fz$ with Dirichlet $ 0$-boundary,
  we have $v=v_\fz$ as desired.

{\it Case Neumann $0$-boundary.}  In this case we may assume in addition  that $\int_{\Omega_k} v_k\,\dd x=0$.
 It suffices to prove $v_k\to v_\fz$ in $C^{1,\alpha}(\Omega_\fz)$ as $\ez\to0$.   Firstly we show that  \eqref{vez1} also holds with some constant $C$ independent of $k$. The proof is much similar to the case Dirichlet $0$-boundary.
We sketch it. First, since $v_k\in W^{1,p}(\Omega)$ is a weak solution to $\mathcal L_{A,\ez,p}v_k=g$ in $\Omega_k$ with Neumann $0$-boundary,
one also has  \eqref{vez2}, and then gets \eqref{vezz}.
Thanks to the assumption $\int_{\Omega_k} v_k\,\dd x=0$ in this case, we could apply the Sobolev-Poincar\'e inequality to get  \eqref{vez3}
with the constant uniformly in $k$, where   note that  $\partial \Omega_k$ are uniform Lipschitz  and has uniform bounded diameters.
We then  choose small $\eta$ to get the desired result.

 Next, it is well-known that,  there exists $\alpha>0$ such that,
for  any smooth subdomain $ U\Subset\Omega_\fz$,
 $v_k\in C^{1,\alpha}(\overline U)$  uniformly in  all $k\in\nn$.
Noting the assumption $\int_{\Omega_k} v_k\,\dd x=0$ in this case,
 we can find  a function $v\in C^{1,\alpha}(\Omega_\fz)$ such that
  $v_k\to v $ and $Dv_k\to Dv$ in $ C^{0,\alpha}(\Omega_\fz)$      as $ \nn\ni k\to\infty$ (up to some subsequence).
In particular, $V_{A ,\ez}(Dv_k)\to V_{A,\ez}(Dv)
$  in $C^0(\Omega_\fz)$  as $ k\to\infty$. In particular $V_{A^\ez,\ez}(Dv_k)\to V_{A^\ez,\ez}(Dv)$ in $C^{0,\alpha}(\Omega_\fz)$.

Moreover, given any $\phi\in W^{1,\fz} (\Omega_\fz)$ we extend it to be a function $\wz \phi\in W^{1,\fz}(\rn)$.
Let $U_m\Subset U_{m+1}\Subset\Omega_\fz$ with $d_H(\Omega_\fz,U_m)\to 0$ as $m\to \fz$. One has
\begin{align*}\int_{\Omega_\fz} V_{A^\ez,\ez}(Dv)\cdot D\phi\,\dd x
&= \lim_{m\to\fz} \int_{U_m} V_{A^\ez,\ez}(Dv )\cdot D\phi\,\dd x&\\
&= \lim_{m\to\fz}\lim_{k\to\fz}
\int_{U_m} V_{A^\ez,\ez}(Dv )\cdot D\phi\,\dd x
\\
&= \lim_{m\to\fz}\lim_{k\to\fz}
[\int_{\Omega_k} V_{A^\ez,\ez}(Dv_k)\cdot D\wz \phi\dd x-\int_{\Omega_k\setminus U_m} V_{A^\ez,\ez}(Dv_k)\cdot D\wz \phi\,\dd x].
\end{align*}
Observe that
\begin{align*}
\int_{\Omega_k} V_{A^\ez,\ez}(Dv_k)\cdot D\wz \phi\dd x
&=\int_{\Omega_k}g\wz \phi \dd x=\int_{\Omega_\fz} g  \phi\,\dd x.
\end{align*}
Since
 $\|\sqrt ADv_k \|_{L^p(\Omega_k)}\le C\|g\|_{L^{p'(\Omega_\fz)}}^{p'/p}$ for all $k\in\nn$, one has
\begin{align*}
\Big|\int_{\Omega_k\setminus U_m } V_{A^\ez,\ez}(Dv_k)\cdot D\phi\dd x\Big|
&\le \|D\phi\|_{L^\fz(\rn)} \| V_{A^\ez,\ez}(Dv_k)\|_{L^1(\Omega_k\setminus U_m)}\\
&\le
\|D\phi\|_{L^\fz(\rn)}\|V_{A^\ez,\ez}(Dv_k)\|_{L^{p'}(\Omega_k)}|\Omega_k\setminus U_m|^{1/p}\to0
\end{align*}
We therefore get
\begin{align*}\int_{\Omega_\fz} V_{A^\ez,\ez}(Dv)\cdot D\phi\,\dd x  =\int_{\Omega_\fz} g\phi\,\dd x.
\end{align*}
Since $W^{1,\fz}(\Omega_\fz)$ is dense in $W^{1,p'}(\Omega_\fz)$,
and $V_{A.\ez}(Dv_k)\in W^{1,p}(\Omega_\fz)$ uniformly in $k$,
 we know that this holds for all $\phi \in W^{1,p}(\Omega)$. Thus  $v$ is a weak solution to
$\mathcal L_{A,\ez,p} v=g$ in $\Omega_\fz$  with  Neumann $0$-boundary.

To see get $v=v_\fz$, it then suffices to show that
 $\int_{ \Omega_\fz}v_\fz\,\dd x=0$. To see this, write
$$
\int_{\Omega_\fz}v_\fz\,\dd x=\lim_{k\to\fz} \int_{\Omega_\fz}v_k\,\dd x=
\lim_{k\to\fz} \Big[\int_{\Omega_k}v_k\,\dd x+  \int_{\Omega_k\setminus \Omega_\fz}v_k\,\dd x\Big]$$
By $\int_{\Omega_k}v_k\,\dd x=0$ and the H\"older inequality, one has
$$
\Big|\int_{\Omega_\fz}v_\fz\,\dd x\Big| =
\lim_{k\to\fz}  \Big| \int_{\Omega_k\setminus \Omega_\fz}v_k\,\dd x\Big|\le  \liminf_{k\to\fz} \|v_k\|_{L^p(\Omega_k)}| \Omega_k\setminus \Omega_\fz|^{1/p'}. $$
By $\int_{\Omega_k}v_k\,\dd x=0$ again and the Sobolev-Poincar\'e inequality, one has $$\|v_k\|_{L^p(\Omega_k)}\le C  \|Dv_k\|_{L^p(\Omega_k)}\le C\|g\|_{L^{p'}(\Omega_\fz)}^{p'/p}.$$
Since $|\Omega_k\setminus \Omega_\fz|\to0$ as $\nn\ni k\to\fz$, we have
$\int_{\Omega_k }v_k\,\dd x=0$ as desired.
\end{proof}

\begin{proof}[Proof of Lemma \ref{approx}]
It suffices to prove $\sqrt {A^\ez}Dv_\ez \to \sqrt A D v_0$ in  $L^p(\Omega)$  as $\ez\to0$.

In the case Dirichlet $0$-boundary, similarly to the proof of Lemma \ref{appconv}, for any $\ez\in(0,1)$ one has
\begin{equation}\label{vezz1}
  \int_{\Omega  }|\sqrt{A^\ez }Dv_\ez|^p\,\dd x \leq C(n,p,L,\Omega )\int_{\Omega}(|g|^{p'}+1)\,\dd x.
\end{equation}
%
Since
 $v_{\ez}-v_0\in W^{1,p}_0(\Omega)$,    one has
\begin{equation}\label{vezz2}
 \int_{\Omega } V_{A^\ez,\ez}(Dv^\ez)\cdot\left( D v_\ez- D v_0\right)\,\dd x
  =\int_{\Omega }g(v^\ez-v_0)\,\dd x, 
\end{equation}
and using $v_{\ez}-v_0$  to replace $v_0$ one also has
\begin{equation}\label{vezz3}
 \int_{\Omega } V_{A^0,0}(Dv_0)\cdot\left( D v_\ez - D v_0 \right)\,\dd x
  =\int_{\Omega }g(v^\ez-v_0)\,\dd x.
\end{equation}
In the case Neumann $0$-boundary, we may further assume that  $\int_\Omega v_\ez\,\dd x=0$ for $\ez\in[0,1]$. Similarly to the proof of Lemma \ref{appconv},  one also has
\eqref{vezz1}, and then \eqref{vezz2} and \eqref{vezz3}.

From
 \eqref{vezz2} and \eqref{vezz3}, it follows that
$$\int_{\Omega } (\sqrt{A^\ez })^{-1}V_{A^\ez,\ez}(Dv_\ez)\cdot (  \sqrt{A^\ez } Dv_\ez-
 \sqrt{A^\ez }  Dv_0 )\,\dd x
  =\int_\Omega  (\sqrt{A^0  })^{-1}V_{A^0}(Dv_0)\cdot ( \sqrt {A^0} Dv_\ez- \sqrt {A^0}Dv_0 )\,\dd x,$$
 and hence
\begin{align*}
&\int_{\Omega } (\sqrt{A^\ez })^{-1}V_{A^\ez,\ez}(Dv_\ez)\cdot (  \sqrt{A^\ez } Dv_\ez-
 \sqrt{A^0}  Dv_0 )\,\dd x \\
&=  \int_{\Omega } (\sqrt{A^\ez })^{-1}V_{A^\ez,\ez}(Dv_\ez)\cdot (  \sqrt{A^\ez }  Dv_0 - \sqrt{A^0 } Dv_0
)\,\dd x\\
&\quad+ \int_{\Omega } (\sqrt{A^0  })^{-1}V_{A^0}(Dv_0)\cdot ( \sqrt {A^0} Dv_\ez- \sqrt {A^0}Dv_0 )\,\dd x.
%
\end{align*}
Moreover, adding both sides with
$$- \int_{\Omega }
 (\sqrt{A^0  })^{-1}V_{A ^0,\ez}(Dv_0)   (\sqrt{A^\ez } D v_\ez-\sqrt{A^0}Dv_0  )\,\dd x$$
we further have
\begin{align}\label{conver-1}
  I&= \int_{\Omega } [(\sqrt{A^\ez })^{-1}V_{A^\ez,\ez}(Dv_\ez)- (\sqrt{A^0  })^{-1}V_{A ^0,\ez}(Dv_0)]\cdot (  \sqrt{A^\ez } Dv_\ez-
 \sqrt{A^0}  Dv_0 )\,\dd x  \nonumber\\
   &= \int_{\Omega } (\sqrt{A^\ez })^{-1}V_{A^\ez,\ez}(Dv_\ez)\cdot (  \sqrt{A^\ez }  Dv_0 - \sqrt{A^0 } Dv_0
)\,\dd x\nonumber\\
&\quad+ \int_{\Omega }[ (\sqrt{A^0  })^{-1} V_{A^0,0}(Dv_0)-   (\sqrt{A^0  })^{-1}V_{A^0 ,\ez}(Dv_0)  ]\cdot\left(\sqrt{A^0}Dv_\ez-\sqrt{A^0}Dv_0\right)\,\dd x\nonumber\\
  &\quad+\int_{\Omega }  (\sqrt{A^0  })^{-1}V_{A^0 ,\ez}(Dv_0)\cdot(\sqrt{A^0 } Dv_\ez-\sqrt{A^\ez} Dv_\ez)\,\dd x\nonumber\\
&=J_1+J_2+J_3.
\end{align}

Now we show that   $I\to0$ as $\ez\to0$.  Indeed,
by the H\"older inequality and \eqref{vezz1},
\begin{align*}
|J_1|
&\le  \|(\sqrt{A^\ez })^{-1}V_{A^\ez,\ez}(Dv_\ez)\|_{L^{p'}(\Omega )}\|(\sqrt {A^\ez}Dv_\ez -\sqrt {A^0})Dv_0   \|_{L^p(\Omega )}\\
& \le  C[1+\| g\|_{L^{p'}(\Omega )}^{p'/p}]  \|(\sqrt {A^\ez}-\sqrt {A^0})Dv_0\|_{L^p(\Omega )}.
\end{align*}
Thanks to $A^\ez\to A^0$ almost everywhere and $A^\ez\in\mathcal E_L(\Omega )$, this yields that
  $J_3\to 0$ as $\ez\to0$.

By the H\"older inequality and \eqref{vezz1}, one has
\begin{align}|J_2|
\leq&C \| g\|_{L^{p'}(\Omega )}^{p'/p}
\| (\sqrt{A^0  })^{-1} V_{A^0,0}(Dv_0)-   (\sqrt{A^0  })^{-1}V_{A^0 ,\ez}(Dv_0)\|_{L^{p'} (\Omega )}.
\end{align}
Observing
\begin{equation}\label{control}
 |(\sqrt{A^0  })^{-1}V_{A^0 ,\ez}(Dv_0)|^{p'} \le |\sqrt {A^0}Dv_0|^p+1
\end{equation}
and $V_ {A^0 ,\ez}(Dv_0) \to V_{A^0 ,0}(Dv_0)$ almost everywhere,
thanks to \eqref{vezz1} we have   $J_2\to0$ as $\ez\to 0$.
By the H\"older inequality again, one has
\begin{align*}
|J_3|& =
\int_{\Omega } [|\sqrt {A^0} Dv_0|^2+\ez]^{\frac{p-2}2}  [(\sqrt{A ^0} - \sqrt{A ^\ez})\sqrt{A ^0}   Dv_0]\cdot Dv_\ez \,\dd x\\
&\le  \||\sqrt {A^0} Dv_0|^2+\ez]^{\frac{p-2}2}  [(\sqrt{A^0 } -\sqrt{A^\ez} )\sqrt {A^0} Dv_0\|_{L^{p'}(\Omega )}\|Dv_\ez\|_{L^p(\Omega )}.
\end{align*}
By \eqref{vezz1} and \eqref{control},  noting $A^\ez, A\in\mathcal E_L(\rn)$ and $A^\ez\to A$ a.\,e. we know that
$J_3\to0$ as $\ez\to0$ as desired.

On the other hand, recall that
$$
(|\xi|^2+|\eta|^2+\ez)^{\frac{p-2}2}|\xi-\eta|^2\le C(p)[(|\xi|^2+\ez)^{\frac{p-2}2}\xi- (|\eta|^2+\ez)^{\frac{p-2}2}\eta]\cdot (\xi-\eta)\quad\forall\xi,\eta\in\rn.
 $$
Applying this to $\xi= \sqrt{A^\ez }Dv_\ez$ and $\eta= \sqrt{A }Dv $ we have
$$ \int_{\Omega }\left(|\sqrt{A^\ez }Dv_\ez|^2+|\sqrt{A}Dv|^2+\ez\right)^{\frac{p-2}{2}}|\sqrt{A^\ez }Dv_\ez-\sqrt{A}Dv|^2\,\dd x\le I\to 0.$$
If $p\ge 2$, this obviously yields
 $\|\sqrt {A^\ez}Dv_\ez-\sqrt ADv\|_{L^p(\Omega )}\to0.$
   If $1<p<2$,  by Holder's inequality,
\begin{align}
  &\int_{\Omega }|\sqrt{A^\ez }Dv_\ez-\sqrt{A}Dv|^p\,\dd x\nonumber\\
  \leq&\left(\int_{\Omega }\left(|\sqrt{A^\ez }Dv_\ez|^2+|\sqrt{A}Dv|^2+\ez\right)^{\frac{p-2}{2}}|\sqrt{A^\ez }\nabla u_\ez-\sqrt{A}Dv|^2\,\dd x\right)^{\frac{p}{2}}\nonumber\\
  &\left(\int_{\Omega }(|\sqrt{A^\ez }Dv_\ez|^2+|\sqrt{A}Dv|^2+\ez)^{\frac{p}{2}}\,\dd x\right)^{\frac{2-p}{2}},
\end{align}
which converges to $ 0$
 as $\ez\to0$.

By this, noting $A^\ez, A\in\mathcal E_L(\rn)$ and $A^\ez\to A$ a.\,e.,write   $$|Dv_\ez- Dv| \le L |\sqrt {A^\ez}Dv_\ez-\sqrt ADv|+ L|(\sqrt {A^\ez}-\sqrt A)Dv|,$$
one has
$\| Dv_\ez- Dv\|_{L^p(\Omega )}\to0.$
\end{proof}

\renewcommand{\thesection}{Appendix A}
 \renewcommand{\thesubsection}{ A }
\newtheorem{lemapp}{Lemma \hspace{-0.15cm}}
\newtheorem{corapp}[lemapp] {Corollary \hspace{-0.15cm}}
\newtheorem{remapp}[lemapp]  {Remark  \hspace{-0.15cm}}
\newtheorem{defnapp}[lemapp]  {Definition  \hspace{-0.15cm}}
\renewcommand{\theequation}{A.\arabic{equation}}

\renewcommand{\thelemapp}{A.\arabic{lemapp}}

\section{Proof of Lemma 7.1}

 To prove lemma 7.1, given any    bounded uniform domain  $\Omega$, below we briefly recall the construction of extension operator $\Lambda: \dot W^{1,q} (\Omega)\to \dot W^{1,q} (\rn)$ by
  Jones \cite{j81} (see also \cite{KZZ}).   For $1\le q<\fz$, denote by $\dot W^{1,q} (\Omega)$  the homogeneous Sobolev space in any domain  $\Omega\subset\rn$, that is, the collection of all function $v\in L^q_\loc(\Omega)$
with its distributional derivative $Dv\in L^q(\Omega)$.

Recall that $\Omega$  is an  $\ez_0
$-uniform domain for some $\ez_0>0$ if
for any $x,y\in\Omega$ one can find a rectifiable curve $\gz:[0,T] \to\Omega$ joining $x,y$ so that
$$\mbox{ $T=\ell(\gz) \le \frac1{\ez_0} |x-y|$ and
$ \dist(\gz(t),\partial\Omega)\ge \ez_0\min\{t,T-t\}\quad\forall t\in[0,T]$}, $$
where $C$ is a constant.
Note that $|\partial\Omega|=0$.   It is well-known that Lipschitz domains are always $\ez_0$-uniform domains,
where $\ez_0$ depends on Lipschitz constant of $\Omega$.
In the case $\Omega$ is convex, $\ez_0$ depends on $\diam\Omega$ and $|\Omega|$.

  Denote by $W_1=\{S_j\}$   the Whitney decomposition of $\Omega$ and $W_2=\{Q_j\}$ as the Whitney decomposition of $(\overline\Omega)^\complement$
as   \cite[Section 2]{KZZ}.
Set also $W_3=\{Q\in W_2, \ell(Q)\le \frac{\ez_0}{16n}\diam\Omega \}$ as \cite[Section 2]{KZZ}.
By Jones and also \cite{KZZ},   any cube $Q\in W_3$  has a reflection cube $Q^\ast\in W_1$ such that 
 $ \ell (Q)\le \ell(Q ^\ast)\le 4\ell (Q)$ and hence   $\dist(Q^\ast,Q)\le C\ell(Q)$ for some constant $C\ge 1$ depending only on $\ez_0$ and $n$.
For any $Q\in W_2\setminus W_3$ we just write $Q^\ast=\Omega$.

Let  $\{\vz_Q\}_{Q\in W_2}$ be a  partition of unit associate to  $W_2$  so that ${\rm supp}\vz_Q\subset\frac{17}{16}Q$.
The  extension operator is then defined by
$$\Lambda v(x)=\left\{\begin{split} &
\sum_{Q\in W_2}(\bint_{Q^\ast}v\,dz)\vz_{Q}&\quad\forall x\in(\overline\Omega)^\complement\\
&\liminf_{r\to0} \bint_{B(x,r)\cap\Omega} v\,dz&\quad\forall x\in\partial \Omega\\
 &v(x) &\quad\forall x\in \Omega\\
\end{split}
\right.$$
Such extension operator is a slight modification of that in \cite{KZZ} and also \cite{j81}.

By some slight modification of the argument by Jones \cite{j81} (see also \cite{KZZ}),   for   $1\le q<\fz$     one has that
$\Lambda :\dot W^{1,q} (\Omega)\to \dot W^{1,q} (\Omega)$ is a   linear bounded
extension  operator, that is,  for any $v\in W^{1,q}(\Omega)$
we have $\Lambda  v\in \dot W^{1,q}(\rn)$ so that $ \Lambda v|_\Omega=v$ and $\|D \Lambda  v\|_{L^q(\rn)}\le C\|Dv\|_{L^q(\Omega)}$
for some $C$ depending on $n,\ez_0$ and $q$.

Moreover, by some slight modification of arguments in \cite{KZZ},
for any $x\in\overline\Omega$ and  $r\le \frac{\ez_0}{16n}\diam\Omega$,
   one has $\|D \Lambda v\|_{L^n(B(x,r))}\le C\|Dv\|_{L^n(\Omega\cap B(\bar x,Cr))}$ .
Indeed,  the choice of
$r$ implies that  $B(x,r)\cap Q=\emptyset$ for any $Q\in W_2\setminus W_3$  and hence one only need to
bound ${\bf H}_{1,1}$ in \cite[P.1422]{KZZ}  and ${\bf H}_{1,2}=0$ and ${\bf H}_2=0$   in \cite[P.1422]{KZZ}.
Thus $\|D \Lambda v\|_{L^n(B(x,r)\setminus \Omega)}\le C\|Dv\|_{L^n(\Omega\cap B(\bar x,Cr))}$.
Moreover, for any  $x\notin\overline\Omega$,  denote by  $\bar x\in\partial \Omega$ is the nearest point of $x$.
If   $\dist(x,\partial\Omega)<r<\diam\Omega$,
one has $$\|D \Lambda v\|_{L^q(B(x,r))}\le \|D \Lambda v\|_{L^q(B(\bar x,2r))} \le C\|Dv\|_{L^q(\Omega\cap B(\bar x,Cr))}.$$


%


\begin{proof}[Proof of Lemma 7.1.]

Let
  $A=(a_{ij})\in \mathcal E_L(\Omega)$ with $DA\in L^q(\Omega)$ with $q\ge n$. Write  $\wz A=(\Lambda a_{ij})$.
By the boundedness of $\Lambda$, we have $\|D\wz A\|_{L^q(\rn)}<C\|DA\|_{L^q(\Omega)}$. Noting
$$
\langle \wz A(x)\xi, \xi\rangle= \sum_{Q\in W_2}(\bint_{Q^\ast}\langle  A(z)\xi,\xi\rangle\,dz) \vz_Q(x)$$
we know that $ \wz A\in \mathcal E_L (\rn)$.
Moreover  in the case $\|DA\|_{L^n(\Omega)}<\fz$,  we have
$ \Phi_{\wz  A}(\Omega_\eta,r)\le C\Phi_{ A}(\Omega, Cr)$ whenever $x\in\Omega_\eta$ and
$0< \eta<r<\diam\Omega$.

For $\ez>0$, $A^\ez= \wz A\ast\eta_\ez$, where $\eta_\ez$ is  the standard smooth  mollifier.
Since $\langle \wz A\ast  \eta (x)\xi, \xi\rangle=\langle \wz A \xi, \xi\rangle \ast\eta (x)$, we know that
$A^\ez\in\mathcal E_L(\rn)$.
 Moreover,
in the case $\|DA\|_{L^n(\Omega)}<\fz$,  we have $\|\wz A\ast \eta_\ez\|_{L^n(B(x,r))}\le \|\wz A \|_{L^n(B(x,r+\ez))}$.
For any $x\in\Omega_\eta$ and
$0< \ez\le \eta <r<\diam\Omega$, we know that  $ \Phi_{ A^\ez }(\Omega_\eta,r)\le   \Phi_{ \wz A   }(\Omega_{\ez+\eta},r+\ez) $.
and hence $ \Phi_{ A^\ez}(\Omega_\eta,r)\le C\Phi_{ A}(\Omega, Cr)$  as desired.
\end{proof}

\noindent Fa Peng,

\noindent
Academy of Mathematics and Systems Science, the Chinese Academy of Sciences, Beijing 100190, P. R. China

\noindent{\it E-mail }:  \texttt{fapeng@amss.ac.cn}

\bigskip

\noindent Qianyun Miao,

\noindent
School of Mathematics and Statistics, Beijing Institute of Technology, Beijing 100081, P. R. China.

\noindent{\it E-mail }:  \texttt{qianyunm@bit.edu.cn}
\bigskip

\noindent Yuan Zhou

\noindent School of Mathematical Sciences, Beijing Normal University, Haidian District Xinejikou Waidajie No.19, Beijing 100875, P. R. China

\noindent {\it E-mail }:
\texttt{yuan.zhou@bnu.edu.cn}



\begin{thebibliography}{99}

\vspace{-0.3cm}
\bibitem{A92}
V. Adolfsson,
$L^2$-integrability of second-order derivatives for Poisson's equation in nonsmooth domains.
Math. Scand. 70 (1992),  146-160.

\vspace{-0.3cm}
\bibitem{ADN59}
S. Agmon, A. Douglis and L. Nirenberg, {\em Estimates near the boundary for solutions of elliptic partial differential equations satisfying general boundary conditions. I.}
Commun. Pure Appl. Math. \textbf{12} (1959), 623--727.



\vspace{-0.3cm}
\bibitem{Bern04}
S. Bernstein, {\em Sur la nature analytique des solutions des \'equations aux d\'eriv\'ees partielles du second ordre}.
Math. Ann. \textbf{59} (1904), 20--76.

\vspace{-0.3cm}
\bibitem{bcdm}
A. K. Balci, A. Cianchi, L. Diening and V. Maz'ya,
{\em
A pointwise differential inequality and second-order regularity for nonlinear elliptic systems},
Mathematische Annalen (2021).


\vspace{-0.3cm}
\bibitem{CMaz11}
A. Cianchi and V. G. Maz'ya,
{\em Global Lipschitz regularity for a class of quasilinear elliptic equations}.
Comm. Partial Differential Equations \textbf{36} (2011), no. 1, 100--133.


\vspace{-0.3cm}
\bibitem{CMaz2017}
A. Cianchi and V. G. Maz'ya,
{\em Quasilinear elliptic problems with general growth and merely integrable, or measure, data}.
Nonlinear Anal.  {\bf 164} (2017), 189--215.

\vspace{-0.3cm}
\bibitem{CMaz}
A. Cianchi and V. G. Maz'ya,
{\em Second-order two-sided estimates in nonlinear elliptic problems}.
Arch. Rational Mech. Anal {\bf 229} (2018), 569-599.


\vspace{-0.3cm}
\bibitem{CMaz2019}
A. Cianchi and V. G. Maz'ya,
{\em Optimal Second-order regularity for the p-Laplace system}.
J. Math. Pures Appl. {\bf 132} (2019), 41-78.

\vspace{-0.3cm}
\bibitem{Grisvard}
P. Grisvard,
{\em Elliptic problems in nonsmooth domains}.
Pitman, Boston, (1985).

\vspace{-0.3cm}
\bibitem{Horm63}
L. H\"ormander, {\em Linear Partial Differential Operators}. Springer, Berlin, (1963).

\vspace{-0.3cm}
\bibitem{j81}
P. W. Jones, {\em Quasiconformal mappings and extendability of functions in Sobolev spaces.} Acta Math. {\bf 147} (2018), 71-88.


\vspace{-0.3cm}
\bibitem{KZZ}
P. Koskela, Y. Zhang and Y. Zhou, {\em Morrey-Sobolev extension domains.} J. Geom. Anal. {\bf 27} (2017),   1413-1434.

\vspace{-0.3cm}
\bibitem{LU1}
 O. A. Ladyzenskaya and N. N. Ural'ceva, {\em Quasilinear elliptic equations and variational problems
with many indepedent variables}.  Usp.Mat.Nauk.16(1961),19-92(Russian);English translation in Russian Math.Surveys 16(1961),17-91.





\vspace{-0.3cm}
\bibitem{LadU68}
O. A. Ladyzenskaya and N. N. Ural'ceva, {\em Linear and quasilinear elliptic equations}. Academic Press, New York, (1968).

\vspace{-0.3cm}
\bibitem{LGM}
Lieberman, G. M.,
{\em The natural generalization of the natural conditions of Ladyzenskaya and Ural'ceva for elliptic equations}.
Commun. Partial. Differ. Equ.  {\bf 16} (1991), 311--361.

\vspace{-0.3cm}
\bibitem{Maz1962}
 V. G. Maz'ya,
{\em The negative spectrum of the higher-dimensional Schr\"{o}dinger operator}.
Dokl. Akad. Nauk SSSR  {\bf 144} (1962)(Russian), 721-722. English translation: Sov. Math. Dokl. {\bf 3} (1962)

\vspace{-0.3cm}
\bibitem{Maz1964}
 V. G. Maz'ya,
{\em On the theory of the higher-dimensional Schr\"{o}dinger operator}.
Izv. Akad. Nauk SSSR Ser. Mat.  {\bf 28} (1964)(Russian), 1145--1172.

\vspace{-0.3cm}
\bibitem{Horm63}
L. H\"ormander, {\em Linear Partial Differential Operators}. Springer, Berlin, (1963).

\vspace{-0.3cm}
\bibitem{HorJ}
R. A. Horn and Ch. R. Johnson,
{\em Matrix Analysis}. Cambridge University Press (1985).



\vspace{-0.3cm}
\bibitem{MazS09}
V. G. Maz'ya and T. O. Shaposhnikova, {\em Theory of Sobolev Multipliers. With Applications to Differential and Integral Operators}. Springer, Berlin, (2009).

\vspace{-0.3cm}
\bibitem{SZ} D. Swanson and W. P. Ziemer, {\em Sobolev functions whose inner trace at the boundary is zero}. Ark. Mat., {\bf 37} (1999), 373-380.

\vspace{-0.3cm}
\bibitem{Shaud34}
J. Schauder, {\em Sur les \'equations lin\'eaires du type elliptique a coefficients continuous}.
C. R. Acad. Sci. Paris \textbf{199} (1934), 1366--1368.

\vspace{-0.3cm}
\bibitem{TG}
G. Talenti,
{\em Nonlinear elliptic equations, rearrangements of functions and Orlicz spaces}.
Ann. Math. Pura Appl.  {\bf 120} (1979), 159-184.


\vspace{-0.3cm}
\bibitem{Tr}
N. S. Trudinger, {\em On harnack type inequalities and their application to quasilinear elliptic equations}. Comm. Pure. Appl. Math.
\textbf{20} (1967), 721-747.


%
%
%
%
%
%
%
%
%
%
%
%
%
%
%
%
%
%
%
%
%
%
%
%
%
%
%
%
%
%
%
%

%
%
%
%
%
%
%
%
%
%
%
%
%
%
%
%
%
%
%
%
%
%
%
%
%
%
%
%

%
\end{thebibliography}
\end{document}